\documentclass[12pt,reqno]{amsart}

\usepackage{amssymb}
\usepackage{etoolbox}
\usepackage{graphicx}
\usepackage{stmaryrd}
\usepackage{color}
\usepackage{pifont}
\usepackage{enumerate}
\DeclareMathOperator*{\esssup}{ess\,sup}
\DeclareMathOperator*{\essinf}{ess\,inf}
\newcommand{\ee}{\varepsilon}

\newcommand{\N}{{\mathbb N}}
\newcommand{\Q}{{\mathbb Q}}
\newcommand{\R}{{\mathbb R}}
\newcommand{\Z}{{\mathbb Z}}

\newcommand{\cH}{{\mathcal H}}
\newcommand{\cI}{\mathcal I}
\newcommand{\cK}{\mathcal K}
\newcommand{\cL}{\mathcal L}
\newcommand{\cM}{\mathcal M}

\newcommand{\cF}{\mathcal F}

\newcommand{\cJ}{\mathcal J}
\newcommand{\cG}{\mathcal G}
\newcommand{\cE}{{\mathcal E}}
\newcommand{\cV}{{\mathcal V}}
\newcommand{\cW}{{\mathcal W}}
\newcommand{\cU}{{\mathcal U}}

\newtheorem{thm}{Theorem}[section]

\newtheorem{lem}[thm]{Lemma}
\newtheorem{cor}[thm]{Corollary}

\newtheorem{prop}[thm]{Proposition}
\newtheorem{ex}[thm]{Example}

\newtheorem{prob}[thm]{Problem}

\newcommand{\inte}{{\mathrm{int}}\,}

\newcommand{\cl}{{\mathrm{cl}}\,}
\newcommand{\conv}{{\mathrm{conv}}\,}
\newcommand{\myspan}{{\mathrm{span}}\,}

\newcommand{\Id}{{\mathrm{Id}}}

\newcommand{\di}{\diamondsuit}

\numberwithin{equation}{section}

\begin{document}
\bigskip

\title{Approximation of rearrangements by polarizations}
\author[Gabriele Bianchi, Richard J. Gardner, Paolo Gronchi, and Markus Kiderlen]{Gabriele Bianchi, Richard J. Gardner, Paolo Gronchi, and Markus Kiderlen}
\address{Dipartimento di Matematica e Informatica ``U. Dini", Universit\`a di Firenze, Viale Morgagni 67/A, Firenze, Italy I-50134} \email{gabriele.bianchi@unifi.it}
\address{Department of Mathematics, Western Washington University,
Bellingham, WA 98225-9063,USA} \email{Richard.Gardner@wwu.edu}
\address{Department of Mathematical Sciences, University of Aarhus,
Ny Munkegade, DK--8000 Aarhus C, Denmark} \email{kiderlen@math.au.dk}
\thanks{First author supported in part by the Gruppo
Nazionale per l'Analisi Matematica, la Probabilit\`a e le loro
Applicazioni (GNAMPA) of the Istituto Nazionale di Alta Matematica (INdAM). Fourth author supported by the Aarhus University Research Foundation grant AUFF-E-2024-6-13.}
\dedicatory{In fond memory of Paolo Gronchi, whose brilliance often showed us the way\\ and whose presence always brought us joy}

\begin{abstract}
The symmetric decreasing rearrangement of functions on $\R^n$ features in several seminal inequalities, such as the P\'olya-Szeg\H{o} inequality.  The latter was shown by the authors \cite{BGGK2} to hold for all smoothing rearrangements, a class that includes the more general $(k,n)$-Steiner rearrangement, as well as others introduced by Brock and by Solynin.  The theory of rearrangements and their associated set maps is developed, with an emphasis on approximation, particularly by polarizations. The P\'olya-Szeg\H{o} inequality holds with equality for polarizations, so is proved relatively easily for rearrangements that can be suitably approximated by them.  One goal here is to show that the Brock rearrangements cannot be approximated in such a way.  It turns out that under mild conditions, each set map associated with a rearrangement has in turn an associated contraction map from $\R^n$ to $\R^n$.  With this new analytical tool, several general results on the approximation of rearrangements are also proved.
\end{abstract}
\maketitle

\section{Introduction}

Steiner symmetrization, introduced by Jacob Steiner around 1836 in the course of his work on the isoperimetric inequality, is the forerunner of many similar processes. Schwarz symmetrization followed some fifty years later, and shortly afterwards, in 1899 according to Talenti in his remarkable survey \cite[Section~2]{Tal} on rearrangements, Carlo Somigliana employed decreasing rearrangements of functions in hydrostatics. Nowadays there are many different types of symmetrization of sets and rearrangements of functions.  They all share the general idea of replacing a set or a function by one which has more desirable or tractable properties, while preserving some important quantity. There are an astounding number of applications to diverse areas, both in mathematics and beyond.  For more information and hints to the literature, we refer to the introductions to our previous papers \cite{BGG1}, \cite{BGG2}, for symmetrization, and to those of \cite{BGGK1}, \cite{BGGK2}, for rearrangements.

The focus here is on rearrangements.  Suppose that $X$ is a class of functions in $\cM(\R^n)$, the measurable functions on $n$-dimensional Euclidean space $\R^n$, that contains the characteristic functions of sets in $\cL^n$, the $\cH^n$-measurable sets of finite measure.  A {\em rearrangement} $T$ on $X$ is a map $T:X\to\cM(\R^n)$ with two properties:  $T$ is monotonic (meaning here monotonic up to sets of $\cH^n$-measure zero) and equimeasurable (preserving the $\cH^n$-measure of superlevel sets).  (Note that the term ``rearrangement" is also commonly used for the image $Tf$ of $f$.) Perhaps the most important example is the {\em symmetric decreasing rearrangement} on $\cV(\R^n)$, the class of nonnegative measurable functions on $\R^n$ that vanish at infinity.  (See Sections~\ref{subsec:notations} and~\ref{Properties} for definitions and terminology.)  This maps $f\in \cV(\R^n)$ to $f^{\#}\in \cV(\R^n)$ (another common notation is $f^{\star}$), where each superlevel set $\{x\in \R^n: f^{\#}(x)> t\}$ is a ball in $\R^n$, of volume equal to $\cH^n(\{x\in \R^n: f(x)> t\})$ and with center at the origin.  In other words, the subgraph $\{(x,t)\in \R^n\times \R: f^{\#}(x)\ge t\}$ of $f^{\#}$ results from applying Schwarz symmetrization in $\R^{n+1}$ (specifically, $(n,n+1)$-Steiner symmetrization, see Example~\ref{ex429}(iii) below, where $H$ is the $x_{n+1}$-axis) to the subgraph $\{(x,t)\in \R^n\times \R: f(x)\ge t\}$ of $f$. Other terms, such as ``spherically symmetric rearrangement" or ``radially decreasing rearrangement" have been used in the literature.

Symmetrals of sets and rearrangements of functions often appear in fundamental inequalities.  For example, a standard version of the P\'olya-Szeg\H{o} inequality states that if $1\le p\le \infty$ and  $f\in W^{1,p}(\R^n)\cap \cV(\R^n)$, then $f^{\#}\in W^{1,p}(\R^n)$ and
$$
\int_{\R^n}\|\nabla f^{\#}(x)\|^p\, dx\leq \int_{\R^n}\|\nabla f(x)\|^p\; dx.
$$
See, e.g., \cite[Theorem~3.20 and~p.~113]{Baer}; when $p=\infty$, the integrals of $p$th powers are replaced by the essential suprema over $\R^n$. In \cite[Corollary~6.5]{BGGK2}, we proved that the same result holds when the symmetric decreasing rearrangement is replaced by any smoothing rearrangement on $\cV(\R^n)$, or equivalently, any that reduces the modulus of continuity of functions in $\cV(\R^n)$.  These terms are defined in Section~\ref{Properties}.  The class of smoothing rearrangements on $\cV(\R^n)$ includes, as well as the symmetric decreasing rearrangement, several others for which the P\'olya-Szeg\H{o} inequality was known to hold.

One particularly simple smoothing rearrangement is called {\em polarization} (or {\em two-point symmetrization}) with respect to an {\em oriented} hyperplane $H$.  See \cite[p.~3]{BGGK1} for historical remarks.  This takes a function $f:\R^n\to \R$ and maps it to
\begin{equation}\label{pol}
P_Hf(x)=\begin{cases}
\max\{f(x),f(x^{\dagger})\},& {\text{if $x\in H^+$,}}\\
\min\{f(x),f(x^{\dagger})\},& {\text{if $x\in H^-$}},
\end{cases}
\end{equation}
where $^{\dagger}$ denotes the reflection in $H$ and where $H^+$, $H^-$, are the two closed half-spaces bounded by $H$ and determined by its orientation.  Considering their very basic definition, polarizations turn out to be amazingly versatile, because a variety of rearrangements can be approximated by them. Several types of approximation are defined in Section~\ref{Approxrearrange}. If $\cU$ and $\cW$ are two families of mappings from $X\subset {\cV}(\R^n)$ to ${\cV}(\R^n)$, we call a map $T\in\cU$ {\em approximable in $L^p$} on $X$ by maps in $\cW$ if there are $T_k\in \cW$, $k\in \N$, such that $\|T_kf-Tf\|_p\to 0$ as $k\to \infty$, for all $f\in X$.  If the maps $T_k$ are allowed to depend on $f\in X$, we say that $T\in\cU$ is {\em weakly approximable in $L^p$} on $X$ by maps in $\cW$.  The word {\em sequentially} is added if $T_k=S_k\circ S_{k-1}\circ\cdots \circ S_1$ for some fixed sequence $(S_k)$ of maps in $\cW$.  Van Schaftingen \cite[Theorem~1]{VS2} improved on earlier work of his  \cite[Theorem~4.4]{VS1} and of Brock and Solynin \cite[Theorem~6.1]{BS} by showing that Steiner and Schwarz ({i.e.~$(k,n)$}-Steiner) rearrangements are sequentially approximable, in an explicit way, in $L^p$ on $L^p_+(\R^n)$, $1\le p<\infty$, by polarizations.  Hajaiej proved the same result independently in \cite{Haj}. Solynin \cite[Lemmas~7.4 and~9.2]{Sol1} proved that Solynin rearrangements (see Example~\ref{ex429}(iv) below) are weakly approximable in $L^p$, $1\le p<\infty$, on $L^p_+(\R^n)$ by finite compositions of polarizations.

An elegant method to establish inequalities involving such approximable rearrangements is to prove them first for polarizations and then use the approximation.  This approach was pioneered by Baernstein and Taylor \cite[Section~2]{BT}, who were also the first to approximate the symmetric decreasing rearrangement by polarizations.  The method can be applied to the P\'olya-Szeg\H{o} inequality, since it holds (actually with equality, see \cite[Proposition~3.12]{Baer}) for polarizations, and this was done by Brock and Solynin \cite[Theorem~8.2]{BS}, Solynin \cite[Theorem~10.2]{Sol1}, and Van Schaftingen \cite[Section~6]{VS1}, for the approximable rearrangements mentioned above.  If all smoothing rearrangements could be suitably approximated by polarizations, our result in \cite{BGGK2} would be relatively easy to prove.  However, we show in Corollary~\ref{corBrockWeakfunctions} that Brock rearrangements (see Example~\ref{ex429}(v) below) cannot even be weakly approximated in $L^p$, on any subfamily of $L^p_+(\R^n)$ containing the $C^{\infty}$ functions with compact support, by finite compositions of polarizations. This proves our conjecture in \cite[p.~6]{BGGK2}, and since this type of approximation seems to be the weakest that would serve the purpose, it also appears, for the first time, to put a limit on the applicability of polarizations in this context.

If $A\subset \R^n$, we may apply \eqref{pol} to its characteristic function $1_A$ and view polarization as a map on sets, defined by
\begin{equation}\label{pol2}
(P_HA)\cap H^+=\left(A\cup A^{\dagger}\right)\cap H^+~~\quad{\text{and}}~~\quad
(P_HA)\cap H^-=\left(A\cap A^{\dagger}\right)\cap H^-,
\end{equation}
where $A^{\dagger}$ is the reflection of $A$ in $H$. See \cite[Figure~1(a)]{BF} for an illustration. It is easy to check that $P_H$ is monotonic and measure preserving on measurable sets. Generally (see Propositions~\ref{fromSetToFct} and~\ref{lemapril30}, versions of which were proved in \cite{BGGK1}) mild conditions ensure that a rearrangement $T$ gives rise to a well-defined, monotonic, and measure-preserving set map $\di_T$, defined by $\di_T A=\{x: T1_A(x)=1\}$, such that
$$
\{x: Tf(x)\ge t\}=\di_T\{x: f(x)\ge t\}\quad{\text{and}}\quad\{x: Tf(x)> t\}=\di_T\{x: f(x)> t\},
$$
essentially, for $t>\essinf f$. (For example, if $T$ is the polarization map defined by \eqref{pol}, then $\di_T$ is the set map defined by \eqref{pol2}.) The rearrangement $T$ is essentially determined by $\di_T$ via \eqref{eqoct62} below, and under these circumstances, the rearrangements and the set maps go hand in hand.

If $B(x,r)$ is the closed ball with center $x\in \R^n$ and radius $r>0$, then $P_HB(x,r)=B(\psi(x),r)$, where $\psi:\R^n\to \R^n$ is the identity on $H^+$ and the reflection in $H$ on $H^-$. The map $\psi$ is a contraction, and it completely describes the action of the polarization $P_H$ on balls. It can be visualized in terms of ``paper-folding" at $H$ and is a useful intuitive tool. Our new results begin by generalizing this observation. In Lemma~\ref{lem11}, we show that if $\di$ is monotonic, measure preserving, and maps balls to balls, then there are contractions $\psi_{\di,r}:\R^n\to \R^n$ such that $\di B(x,r)=B(\psi_{\di,r}(x),r)$, for all  $x\in \R^n$ and $r>0$.  Moreover (see Lemma~\ref{lem11_inependent}), the contraction is independent of $r$ when $\di$ is smoothing. It is also independent of $r$ when $\di$  respects $H$-cylinders, i.e., is such that if an infinite spherical cylinder orthogonal to $H$ contains a set, it must also contain its image, essentially; in this case, $\psi_\di$ is constant on each hyperplane parallel to $H$ and thus reduces to a contraction $\varphi_\di:\R\to\R$. The latter was introduced and analyzed in \cite[Lemma~6.2]{BGGK1}.  These contractions not only give complete control over the images under $\di$ of balls, but also provide very useful analytical interpretations of the underlying geometry. Example~\ref{ex310} provides explicit formulas for the contractions associated with the set maps of principal interest here, those from Example~\ref{ex429}.

In Section~\ref{Approxrearrange}, the various types of approximation are defined and the previously known results for the standard rearrangements are summarized.  The main new results are Lemma~\ref{lemLpclosure} and Theorem~\ref{thm514}. The former gives conditions ensuring that the approximability of a rearrangement in $L^p$ on a class $X$ can be extended to the $L^p$ closure of $X$, while the latter specifies how the approximability of a rearrangement relates to that of its associated set map. With Theorem~\ref{thm514} in hand,  Lemma~\ref{Markuslem} shows that in contrast to Van Schaftingen's result for Steiner and Schwarz rearrangements mentioned above, Brock and Solynin rearrangements are not sequentially approximable in $L^p$ by polarizations.

Let $H$ be a fixed hyperplane in $\R^n$.  The focus in Section~\ref{Approximation} is on the class $\cJ_H(\cE)$ of
monotonic and measure-preserving maps $\di:\cE\to\cL^n$, where $\cE$ contains all convex bodies, that map balls to balls and respect $H$-cylinders.  The class $\cJ_H({\cL}^n)$ contains all the set maps of primary interest in this paper.  We introduce a special class $\cI$ of contractions that is closed under composition and pointwise convergence, and which contains (see Corollary~\ref{cor313}) $\varphi_{\di}$ for each finite composition $\di$ of polarizations with respect to hyperplanes parallel to $H$.  The key result is Theorem~\ref{thmgpPrime}, which says that if  $\di\in \cJ_H(\cK_n^n)$ and $\varphi_{\di}\not\in \cI$, then $\di$ cannot be weakly approximated on ${\cK}^n_n$ by finite compositions of polarizations with respect to hyperplanes parallel to $H$.  The fact (Lemma~\ref{lem314}) that a contraction $\varphi:\R\to\R$ is not in $\cI$ if there exists a point where $\varphi'$ exists and is different from 0, 1, and -1, allows us to conclude in Corollary~\ref{cor314} that the Brock set map with parameter $0<b<1$ cannot be weakly approximated on ${\cK}^n_n$ by finite compositions of polarizations with respect to hyperplanes parallel to $H$.

The goal of Section~\ref{arb} is to obtain information on approximation by finite compositions of arbitrary polarizations, not just those taken with respect to parallel hyperplanes.  Here our new contractions $\psi_{\di}:\R^n\to\R^n$ come into play.  A crucial observation, proved in Theorem~\ref{lem5}, is that if a contraction $\psi:\R^n\to\R^n$ can be approximated (pointwise) on $B^n$ by finite compositions of contractions associated with polarizations, and the restriction of $\psi$ to the unit ball $B^n$ is injective, then $\psi$ preserves the volume of $B^n$.  The immediate payoff (see Corollary~\ref{cor42}) is that the Brock set map $\di_{B_H}$ with parameter $0<b<1$ cannot be approximated on ${\cK}^n_n$ by finite compositions of polarizations. Theorem~\ref{lem5} is also used in proving Corollary~\ref{corBrockWeak}, which implies that $\di_{B_H}$ with parameter $0<b<1$ cannot be weakly approximated on ${\cF}_n^n$, the class of finite unions of balls, by finite compositions of polarizations.  The latter result requires a good deal more work besides, but together with Lemma~\ref{lemLpclosure} and Theorem~\ref{thm514}, it yields Corollary~\ref{corBrockWeakfunctions}, the result mentioned above concerning Brock rearrangements.

Section~\ref{Solynin} is devoted to Solynin set maps and rearrangements.  The main result, Corollary~\ref{gabcor}, shows that the Solynin rearrangement $So_H$ with respect to a hyperplane $H$ is approximable in $L^p$ on $L^p_+(\R^n)$ by finite compositions of polarizations and hence, by Theorem~\ref{thm514}, the associated set map $\di_{So_H}$ is approximable on ${\mathcal{L}}^n$ by finite compositions of polarizations.

We do not know if Solynin set maps and rearrangements are weakly sequentially approximable by polarizations.  This and a few other open questions are posed in Section~\ref{Questions}.

We are very grateful to a referee of our paper \cite{BGGK2} who drew our attention to Brock's articles \cite{Bro1} and \cite{Bro2}.

\section{Preliminaries}\label{subsec:notations}

As usual, $S^{n-1}$ denotes the unit sphere and $o$ the origin in Euclidean $n$-space $\R^n$.  Unless stated otherwise, we assume throughout that $n\ge 2$.   The standard orthonormal basis for $\R^n$ is $\{e_1,\dots,e_n\}$ and the Euclidean norm is denoted by $\|\cdot\|$.  The term {\em ball} in $\R^n$ will always mean a {\em closed} $n$-dimensional ball unless otherwise stated.  The unit ball in $\R^n$ will be denoted by $B^n$ and $B(x,r)$ is the ball with center $x$ and radius $r$.  If $x,y\in \R^n$, we write $x\cdot y$ for the inner product and $[x,y]$ for the line segment with endpoints $x$ and $y$. If $x\in \R^n\setminus\{o\}$, then $x^{\perp}$ is the $(n-1)$-dimensional subspace orthogonal to $x$. Throughout the paper, the term {\em subspace} means a linear subspace.

If $A$ is a set,  we denote by $\cl A$, $\inte A$, and $\dim A$ the {\it closure}, {\it interior}, and {\it dimension} (that is, the dimension of the affine hull) of $A$, respectively.  If $H$ is a subspace of $\R^n$, then $A|H$ is the (orthogonal) projection of $A$ on $H$ and $x|H$ is the projection of a vector $x\in \R^n$ on $H$.

If $A$ and $B$ are sets in $\R^n$ and $t\in \R$, then we denote by $tA=\{tx:x\in A\}$ the {\em dilate} of $A$ by the factor $t$, and by
$$A+B=\{x+y: x\in A, y\in B\}$$
the {\em Minkowski sum} of $A$ and $B$.  We write $-A=(-1)A$ for the reflection of $A$ in the origin and call $A$ {\em origin symmetric} or {\it $o$-symmetric} if $-A=A$.

If $H$ is an oriented hyperplane in $\R^n$, then $H^+$ and $H^-$ denote the two closed halfspaces bounded by $H$ and determined by the orientation.

We write ${\mathcal H}^k$ for $k$-dimensional Hausdorff measure in $\R^n$, where $k\in\{1,\dots, n\}$.  When dealing with relationships between sets in $\R^n$ or functions on $\R^n$, the term {\em essentially} means up to a set of $\cH^n$-measure zero. By $\kappa_n$ we denote the volume ${\mathcal H}^n(B^n)$ of the unit ball in $\R^n$.

We denote by ${\mathcal C}^n$, ${\mathcal B}^n$, ${\mathcal M}^n$, and ${\mathcal L}^n$ the class of nonempty compact sets, bounded Borel sets, ${\mathcal H}^n$-measurable sets, and ${\mathcal H}^n$-measurable sets of finite ${\mathcal H}^n$-measure, respectively, in $\R^n$.  The family of finite unions of balls in $\R^n$ will be denoted by $\cF_n^n$.

Let ${\mathcal K}^n$ be the class of nonempty compact convex subsets of $\R^n$ and let ${\mathcal K}^n_n$ be the class of {\em convex bodies}, i.e., members of ${\mathcal K}^n$ with interior points. The texts by Gruber \cite{Gru07} and Schneider \cite{Sch14} contain a wealth of useful information about convex sets and related concepts.

Let ${\mathcal{M}}(\R^n)$ (or ${\mathcal{M}}_+(\R^n)$) denote the set of real-valued (or nonnegative, respectively) measurable functions on $\R^n$ and let ${\mathcal{V}}(\R^n)$ denote the set of functions $f \in {\mathcal{M}}_+(\R^n)$ such that ${\mathcal{H}}^{n}(\{x:f(x)>t\})<\infty$ for all $t>0$.  Members of ${\mathcal{V}}(\R^n)$ are often said to {\em vanish at infinity}. For $1\le p< \infty$, let $L^p_+(\R^n)=\{f\in L^p(\R^n):f\ge 0\}$ and note that $L^p_+(\R^n)\subset \cV(\R^n)$.

Our notation for Sobolev spaces such as $W^{1,p}(\R^n)$ is standard. Definitions can be found in many texts, such as \cite{LL}.

If $T:X\to \cM(\R^n)$, where $X\subset \cV(\R^n)$, we shall usually write $Tf$ instead of $T(f)$.  If $T_0,T_1:X\to \cM(\R^n)$ are maps, we say that $T_0$ is {\em essentially equal} to $T_1$ if for $f\in X$, $T_0f(x)=T_1f(x)$ for ${\mathcal{H}}^{n}$-almost all $x\in \R^n$, where the exceptional set may depend on $f$.

If $A\in \cM^n$, then
$$
\Theta(A,x)=\lim_{r\to0^+}\frac{\cH^{n}(A\cap B(x,r))}{\cH^n(B(x,r))},
$$
is the {\em density} of $A$ at $x$, provided the limit exists.  Elements of the set
$$
A^*=\{x\in \R^n: \Theta(A,x)=1\}.
$$
are called {\em Lebesgue density points}, or simply density points, of $A$.  Note that $A^*=A$, essentially, by the Lebesgue density theorem (see, e.g., \cite[Theorem~1.5.2]{Pfe}).

\section{Properties of maps}\label{Properties}

Let $i\in \{1,\dots,n-1\}$ and let $H$ be a fixed $i$-dimensional subspace in $\R^n$.  For ${\mathcal{E}}\subset {\cL}^n$, we consider a map $\di:{\mathcal{E}}\to {\mathcal{L}}^n$ and define
\begin{equation}\label{eqaug31}
\di^*A=(\di A)^*
\end{equation}
for each $A\in {\mathcal{E}}$.  We assume (here and throughout the paper) that the properties listed below hold for all $A,B\in {\mathcal{E}}$.

\begin{enumerate}
        \item[1.]  ({\em Monotonic}) \quad If $A\subset B$, essentially, then $\di A\subset \di B$, essentially.
		\item[2.]  ({\em Measure preserving}) \quad ${\mathcal H}^n(\di A)={\mathcal H}^n(A)$.
		\item[3.] ({\em Maps balls to balls})\quad If $K=B(x,r)$, essentially,  then $\di K=B(x',r')$, essentially.
		\item[4.]  ({\em Respects $H$-cylinders})\quad If $A\subset (B(x,r)\cap H)+H^{\perp}$, essentially, then $\di A\subset (B(x,r)\cap H)+H^{\perp}$, essentially.
		\item[5.]  ({\em Smoothing}) \quad Whenever $d>0$,
		\begin{equation}\label{eq:smothDef}
			(\di^*A)+d B^n\subset \di^*(A+d B^n)=\di(A+d B^n),
		\end{equation}
		essentially, for each bounded $A\in \cE$ with $A+dB^n\in {\mathcal{E}}$, where $\di^*A$ is defined by \eqref{eqaug31}.
\end{enumerate}

Note that in \cite{BGGK1}, the word ``essentially" was omitted from Properties~1, 3, and 4, while in \cite{BGGK2}, Property~4 was not utilized.

We shall denote by $\cJ(\cE)$ the family of monotonic and measure-preserving maps from $\cE$ to $\cL^n$ that map balls to balls.  If $H$ is a hyperplane in $\R^n$, then $\cJ_H(\cE)\subset \cJ(\cE)$ is the subfamily of maps that respect $H$-cylinders.

Information concerning relations between the first three properties listed above and others besides may be found in \cite[Sections~3 and~6]{BGGK1}. The term ``smoothing" was employed by Sarvas \cite[p.~11]{Sar72}, although his definition differs slightly from ours.

In the definition of smoothing, one can equivalently require a pointwise inclusion in \eqref{eq:smothDef}; see \cite[p.~12]{BGGK2}.

Let $X\subset {\cV}(\R^n)$ and let $\cE\subset \cL^n$ be the class of sets with characteristic functions in $X$. For a map $T:X\to  {\cV}(\R^n)$, the induced set map $\di_T:\cE\to \cL^n$ given by
\begin{equation}\label{eqIndic}
\di_T A=\{x: T1_A(x)=1\},
\end{equation}
$A\in \cE$, is well defined since $T1_A\in {\cV}(\R^n)$ and hence  $\cH^n(\di_T A)\le \cH^n(\{x: T1_A(x)>1/2\})<\infty$.

If $X\subset {\cV}(\R^n)$ is as above, we consider the following properties of a map $T:X\to {\cV}(\R^n)$, where the first four properties are assumed to hold for all $f,g\in X$:

\begin{enumerate}
		\item[1.] ({\em Equimeasurable})\quad
		$
		{\mathcal{H}}^{n}(\{x: Tf(x)>t\})={\mathcal{H}}^{n}(\{x: f(x)>t\})
		$
		for $t\in \R$.
		\item[2.] ({\em Monotonic}) \quad $f\le g$, essentially, implies $Tf\le Tg$, essentially.
		\item[3.]  ({\em $L^p$-contracting}) \quad $\|Tf-Tg\|_p\le \|f-g\|_p$ when $f-g\in L^p(\R^n)$.
		\item[4.]   ({\em Modulus of continuity reducing})  If $d>0$, we define the {\em modulus of continuity} of $f\in X$ by
$$
\omega_{d}(f)=\esssup_{\|x-y\|\le d}|f(x)-f(y)|
$$
and say that $T$ {\em reduces the modulus of continuity} if $\omega_{d}(Tf)\le \omega_{d}(f)$ for all $d>0$ and $f\in X$.
		\item[5.] ({\em Smoothing}) \quad We say that $T$ is {\em smoothing} if $\cE=\mathcal L^n$ and the induced map $\di_T$ is smoothing, i.e.,
$$
(\di_T^*A)+d B^n\subset \di_T^*(A+d B^n)=\di_T(A+d B^n),
$$
essentially, for each $d>0$ and $A\in \mathcal B^n$.
\end{enumerate}

The map $T$ is called a {\em rearrangement} if it is equimeasurable and monotonic.

Our approach to set maps and rearrangements of functions, including our definition of the term ``smoothing," differs from that in the groundbreaking works of Brock and Solynin \cite{BS}, Van Schaftingen \cite{VSPhD}, and Van Schaftingen and Willhem \cite{VSW}.  A comparison of these differences in \cite[Appendix]{BGGK1} and \cite[Section~3]{BGGK2} shows that ultimately these approaches are equivalent.

Each rearrangement $T:\cV(\R^n)\to\cV(\R^n)$ is $L^p$-contracting for  $1\le p<\infty$; see \cite[Proposition~3.7 and Appendix~A]{BGGK2} for a proof in our setting. For the convenience of the reader, we now state, in a slightly different form, two results proved in \cite{BGGK1}.

\begin{prop}\label{fromSetToFct}
Let $X\subset {\cV}(\R^n)$, let $\cE=\{A\in {\cL}^n: 1_A\in X\}$, and let $T:X\to  {\cV}(\R^n)$ be equimeasurable. The induced map $\di_T:\cE\to \cL^n$ given by \eqref{eqIndic} is measure preserving and satisfies
\begin{equation}\label{eqoct72}
T1_A=1_{\di_TA},
\end{equation}
essentially, for all $A\in \cE$.
\end{prop}

\begin{proof}
The result is stated for $T:\cV(\R^n)\to \cV(\R^n)$ in \cite[Lemma~4.5]{BGGK1}, but its proof extends to the slightly more general setting here.
\end{proof}

The following proposition is a special case of \cite[Lemma~4.8]{BGGK1}.

\begin{prop}\label{lemapril30}
Let $T:{\cV}(\R^n)\to{\cV}(\R^n)$ be a rearrangement.

\noindent{\rm{(i)}}  The map  $\di_T:\cL^n\to \cL^n$ defined by \eqref{eqIndic} is monotonic.

\noindent{\rm{(ii)}} If $f\in {\cV}(\R^n)$, then
$$
\{x: Tf(x)\ge t\}=\di_T\{x: f(x)\ge t\}\quad{\text{and}}\quad\{x: Tf(x)> t\}=\di_T\{x: f(x)> t\},
$$
essentially, for $t>0$.  Moreover, $T$ is essentially determined by $\di_T$, since
\begin{equation}\label{eqoct62}
Tf(x)=\max\left\{\sup\{t\in \Q,\, t>\essinf f: x\in \di_T \{z: f(z)\ge t\}\},\essinf f\right\},
\end{equation}
essentially.
\end{prop}

The following proposition addresses the extendability of a set map defined on the compact sets.  We do not require the outer regularity assumed in earlier similar extensions, such as that of Brock and Solynin \cite[Section~3]{BS}.  In contrast, \cite[Example~6.8]{BGGK1} shows that the polarization set map $\di_{P_H}$ has more than one essentially different extension to from ${\cK}^n_n$ to ${\cL}^n$.

\begin{prop}
Let $\di:{\mathcal{C}}^n \to \cL^n$ be monotonic and measure preserving. Then there is an essentially unique monotonic and measure-preserving map $\overline{\di}: \cL^n\to\cL^n$ such that $\overline{\di}=\di$ on ${\mathcal{C}}^n $.
\end{prop}

\begin{proof}
Let $A\in \cL^n$.  Since $\cH^n$ is inner regular, there is an increasing sequence $(C_k)$ of sets in ${\mathcal{C}}^n$ such that $\cH^n(A\setminus C_k)\to 0$ as $k\to\infty$. Define $\overline{\di} A=\cup_k \di C_k$.  The continuity of $\cH^n$ from below yields $\cH^n\left(\overline{\di}A\right)=\lim_k \cH^n(C_k)=\cH^n(A)$, so $\overline{\di}$ preserves measure and in particular $\overline{\di} A\in \cL^n$.

We claim that the definition of $\overline{\di}$ is essentially independent of the sequence $(C_k)$.  To see this, let $(C_k\!')$ be another sequence with the same properties, and let $\overline{\di}\,' A=\cup_k \di C_k\!'$. Let $D_k=C_k\cap C_k\!'$.  Then
$$\cH^n(A)-\cH^n(D_k)=\cH^n(A\setminus D_k)\le \cH^n(A\setminus C_k)+\cH^n(A\setminus C_k\!')\to 0$$
as $k\to \infty$. The monotonicity of $\di$ implies that $\di D_k\subset \di C_k$, essentially, so using the measure-preserving property of $\di$, we obtain
\[
\cH^n(\di C_k\setminus \di D_k)=\cH^n(\di C_k)-\cH^n(\di D_k)=\cH^n(C_k)-\cH^n(D_k)\to \cH^n(A)-\cH^n(A)=0,
\]
as $k\to\infty$.  Moreover,
$$\cH^n\left(\overline{\di} A\setminus\di C_k\right)=\cH^n\left(\overline{\di} A\right)-\cH^n(\di C_k)\to 0$$
as $k\to\infty$.  Therefore
\begin{align*}
\cH^n\left(\overline{\di}A\setminus \overline{\di}\,'A\right)&\le \cH^n\left(\overline{\di} A\setminus \di D_k\right)\le \cH^n\left(\overline{\di} A\setminus \di C_k\right)+\cH^n(\di C_k\setminus \di D_k)\to 0
\end{align*}
as $k\to \infty$.  Thus $\cH^n\left(\overline{\di}A\setminus \overline{\di}\,'A\right)=0$, and similarly $\cH^n\left(\overline{\di}\,'A\setminus \overline{\di}A\right)=0$, proving the claim.

It is easy to see that $\overline{\di}$ is monotonic.
\end{proof}

We end this section with some examples of smoothing set maps and rearrangements.

\begin{ex}\label{ex429}
{\em Let $H$ be a hyperplane in $\R^n$.  Each of the following set maps can be defined as a monotonic, measure preserving, and smoothing map from $\cL^n$ to $\cL^n$ and gives rise to a smoothing rearrangement from $\cV(\R^n)$ to $\cV(\R^n)$.

(i) The {\em reflection} of a set $A\subset\R^n$ in $H$ is $R_HA=A^{\dagger}=\{x^{\dagger}: x\in A\}$, where $x^{\dagger}=2(x|H)-x$ the reflection of $x$ in $H$.  The reflection $R_Hf=f^{\dagger}$ in $H$ of a function $f$ on $\R^n$ is defined by $R_Hf(x)=f^{\dagger}(x)=f(x^{\dagger})$.

(ii) When $H$ is oriented, polarization of sets and functions was defined above; see \eqref{pol2} and \eqref{pol}.  For background and references, see \cite[p.~3]{BGGK1}. Proofs of the smoothing property (recall that this is equivalent to reducing the modulus of continuity of functions) are given in \cite[Theorem~1.37]{Baer} and \cite[Lemma~5.1]{BS}.

(iii) If $A\in {\mathcal{L}}^n$, the {\em Steiner symmetral} of $A$ with respect to $H$ is the set $S_HA$ such that for each line $G$ orthogonal to $H$ and meeting $A$, the set $G\cap S_HA$ is a (possibly degenerate) closed line segment with midpoint in $H$ and length ${\mathcal{H}}^1(G\cap A)$, if ${\mathcal{H}}^1(G\cap A)<\infty$, and $G\cap S_HA=\emptyset$ if ${\mathcal{H}}^1(G\cap A)=\infty$ or if $G\cap A$ is not ${\mathcal{H}}^1$-measurable.  If instead $H$ is an $(n-k)$-dimensional subspace in $\R^n$, the {\em Schwarz} or {\em $(k,n)$-Steiner symmetral} of $A$ with respect to $H$ is obtained in the same way, where $G$ is now a $k$-dimensional plane orthogonal to $H$, $G\cap S_HA$ is a (possibly degenerate) $k$-dimensional ball with center in $H$, and ${\mathcal{H}}^1$ is replaced by ${\mathcal{H}}^k$. For the definitions of the corresponding rearrangements of functions in ${\cV}(\R^n)$, we refer to \cite[Definition~6.2]{Baer}.  The smoothing property is proved in \cite[Theorem~6.10]{Baer} and \cite[Corollary~6.1]{BS}.

(iv) If $H$ is a hyperplane in $\R^n$, what we call {\em Solynin} set maps $\di_{So_H}$ and rearrangements $So_H$ are
far-reaching extensions of a process discovered by McNabb \cite{McN}, in which a convex body is transformed continuously into its Steiner symmetral.  In this special case, the action of $\di_{So_H}$ on a convex body $K\subset\R^n$ can be visualized as follows.  If $H$ is horizontal and oriented with the $x_n$-axis, then any chord of $K$ orthogonal to $H$ is left unchanged if its midpoint lies above or on $H$, while if it lies below $H$, it is translated orthogonal to $H$ so that its midpoint belongs to $H$.  If $H$ is imagined to be moving up vertically through $K$, then $\di_{So_H} K$ changes continuously from $K$ itself to the Steiner symmetral $S_HK$.  The first extension to non-convex sets was carried out by Solynin \cite{Sol90}, who called it {\em continuous symmetrization}.  It achieved its full generalization in Solynin's paper \cite{Sol1}, where it is analysed in detail and applied to integral inequalities.  Solynin set maps and rearrangements with respect to $H$ act fiber-wise on lines orthogonal to $H$, so it suffices to define them when $n=1$.  In this case, if $H=\{t\}$, with the usual orientation of $\R$, and $A\in {\cL}^1$, then
\begin{equation}\label{Soldef}
\di_{So_{\{t\}} A}=[t-r_A,t+r_A]\cup (A\cap[t,\infty)),
\end{equation}
where $r_A\geq0$ is such that ${\cH}^1(So_{\{t\}} A)={\cH}^1(A)$.  Note that when $A\subset (t,\infty)$, so that $r_A=0$, we  have $\di_{So_{\{t\}} A}=\{t\}\cup A\neq A$, but in this case $\di_{So_{\{t\}} A}=A$, essentially.  Moreover, \eqref{Soldef} essentially agrees with the description above when $A$ is a convex body.  Up to translation, the set $A$ tranforms continuously into its Steiner symmetral $S_{\{o\}}A$ with respect to $\{o\}$ as $t$ goes from $-\infty$ to $\infty$. Compare \cite[(4.5) and (4.6)]{Sol1}, which lead to the definitions
\cite[Definitions~4.3 and~4.4]{Sol1}. The smoothing property is proved in \cite[Lemma~4.4]{Sol1}. An extension analogous to $(k,n)$-Steiner symmetrization is studied in \cite[Section~9]{Sol1}.

(v) For our purposes, {\em Brock} set maps and rearrangements of functions depend on a parameter $0\le b\le 1$.  For convex bodies, the set maps were introduced by P\'olya and Szeg\H{o} \cite[Note~B]{PS} (in collaboration with M.~Shiffman), who used the term {\em continuous symmetrization} also adopted by Solynin.  In this special case, if $H$ is a horizontal hyperplane at height $t_0\in \R$ and $K\subset \R^n$ is a convex body, its image under the Brock set map $\di_{B_H}$ with parameter $b$ is obtained by translating orthogonal to $H$ each vertical chord of $K$ with midpoint at height $t$ so that its midpoint is at height $(1-b)t_0+bt$.  Thus $b=1$ corresponds to the identity map and $b=0$ to Steiner symmetrization with respect to $H$.  The extension of $\di_{B_H}$ to a map on $\cL^n$, leading to the definition of the Brock rearrangement $B_H$ on functions in ${\cV}(\R^n)$, requires considerable ingenuity.  It is described and fully analysed by Brock \cite[Sections~2 and~3]{Bro1} and \cite[Section~2]{Bro2} (where the parametrization is different), and the smoothing property is proved in \cite[Remark~2.3]{Bro2}.  The extension itself was anticipated by Rogers \cite{Rog57}, who used it in his remarkable proof of the Brascamp-Lieb-Rogers inequality.  The crucial step is to define $\di_{B_H}A$ when $n=1$ and $A$ is a finite union of line segments; see Step 1 in \cite[pp.~165--166]{Bro2} and \cite[pp.~106--107]{Rog57}.
\qed}
\end{ex}

\section{Contractions associated with set maps}\label{Contractions}

The following lemma was proved in \cite[Lemma~6.2]{BGGK1} when $t_0=0$ and ${\mathcal{K}}^n_n\subset {\mathcal{E}}$.  We shall need the slight generalization presented here in the proof of Theorem~\ref{thm:not_weaklyfiniteBalls}.

\begin{lem}\label{lemm5}
Let $H=u^{\perp}+t_0 u$, $u\in S^{n-1}$, $t_0\in \R$, let ${\mathcal{E}}\subset {\mathcal{L}}^n$ be a class that contains all balls in $\R^n$, and suppose that $\di:{\mathcal{E}}\rightarrow {\mathcal{L}}^n$ is monotonic, respects $H$-cylinders, and maps balls to balls.  Then there is a contraction (i.e., a Lipschitz function with Lipschitz constant $1$) $\varphi_{\di}:\R\to\R$ such that
\begin{equation}\label{ball}
\di B(x+tu,r)=B\left(x+\varphi_{\di}(t)u,r\right),
\end{equation}
essentially, for $x\in u^\perp$, $t\in \R$, and $r>0$.
\end{lem}

\begin{proof}
It is a simple matter to modify the proof of \cite[Lemma~6.2]{BGGK1} so that it applies when $t_0\in \R$.  Apart from this, the only adjustment required is in the penultimate sentence, where $K$ can be replaced by the ball $B_1$.
\end{proof}

The following result generalizes and extends Lemma~\ref{lemm5} to maps not necessarily respecting $H$-cylinders.

\begin{lem}\label{lem11}
Let ${\mathcal{E}}\subset {\mathcal{L}}^n$ be a class that contains all balls in $\R^n$, and let $\di:\cE\to\cL^n$ be monotonic, measure preserving, and map balls to balls.

\noindent{\rm{(i)}} There are contractions $\psi_{\di,r}:\R^n\to \R^n$ such that
\begin{align}\label{eq:psi1}
\di B(x,r)=B(\psi_{\di,r}(x),r)
\end{align}
for all $x\in \R^n$ and $r>0$. We have
\begin{equation}\label{eq:pseudoContract}
\|\psi_{\di,r}(x)-\psi_{\di,r'}(x')\|\le \max\{\|x-x'\|,|r-r'|\}.
\end{equation}
for $x,x'\in \R^n$ and $r,r'>0$.

\noindent{\rm{(ii)}} Let $H=u^\perp+t_0u$, $u\in S^{n-1}$, $t_0\in\R$, be a hyperplane. If, in addition, $\di$ respects
$H$-cylinders, then $\psi_{\di,r}$ is independent of $r>0$, and
\begin{equation}\label{eq:omregn}
\psi_{\di,r}(x)=x|u^\perp+\varphi_{\di}(\langle x,u\rangle)u
\quad{\text{and}}~~\quad
\varphi_\di(t)=\langle \psi_\di(tu),u\rangle
\end{equation}
for $x\in \R^n$ and $t\in\R$, where $\varphi_{\di}:\R\to\R$ is the contraction associated with $\di$ from \eqref{ball}.
\end{lem}

\begin{proof}
(i) As $\di$ preserves measure and maps balls to balls, for every $x\in \R^n$ and $r>0$, there is a $z_{x,r}\in \R^n$ such that $\di B(x,r)=B(z_{x,r},r)$.  Define $\psi_{\di,r}:\R^n\to \R^n$ by $\psi_{\di,r}(x)=z_{x,r}$.

Let $x,x'\in \R^n$ and $r,r'>0$ be given. If neither of the balls $B(x,r)$ and $B(x',r')$ is contained in the other, the smallest ball $B$ that contains both has diameter $d=\|x-x'\|+r+r'$.  Moreover, $\|x-x'\|>|r-r'|$.
Monotonicity and the measure-preserving property imply that
$\di B$ is a ball of radius $d$ that contains $\di B(x,r)=B(\psi_{\di,r}(x),r)$ and $\di B(x',r')=B(\psi_{\di,r'}(x'),r')$, so
\[
\|\psi_{\di,r}(x)-\psi_{\di,r'}(x')\|+r+r'\le d= \|x-x'\|+r+r'.
\]
If $B(x,r)$ is contained in $B(x',r')$, say, the smallest ball $B$ containing them both is $B(x',r')$, and $\|x-x'\|\le r'-r$.  Then $B(\psi_{\di,r}(x),r)\subset B(\psi_{\di,r'}(x'),r')$ and hence
$\|\psi_{\di,r}(x)-\psi_{\di,r'}(x')\|\le  r'-r$. Combining these two cases gives \eqref{eq:pseudoContract}.  Letting $r=r'$ in \eqref{eq:pseudoContract}, we see that $\psi_{\di,r}$ is a contraction.

(ii) If $\di$ also respects $H$-cylinders, then Lemma~\ref{lemm5} gives
\begin{align*}
\di B(y+tu,r)=B\big(y+\varphi_{\di}(t)u,r\big)
\end{align*}
for all $y\in u^\perp$, $t\in \R$, and $r>0$. With \eqref{eq:psi1} and $x=y+tu$, we obtain
$$B(\psi_{\di,r}(x),r)=\di B(x,r)=
B\big(x|u^\perp+\varphi_{\di}(\langle x,u\rangle) u,r\big)$$
for all $x\in \R^n$.  Hence,
\[
\psi_{\di,r}(x)=x|u^\perp+\varphi_{\di}(\langle x,u\rangle)u
\]
is independent of $r>0$ and \eqref{eq:omregn} holds.
\end{proof}

Without additional assumptions, the contractions $\psi_{\di,r}$ from \eqref{eq:psi1} may depend on $r$, as the following example shows.

\begin{ex}
{\em Let $u\in S^{n-1}$ be fixed and define $\di:\cK_n^n\to \cK_n^n$ by
\[
\di K=B(ru,r),
\]
where $r>0$ is chosen such that $\cH^n(B(ru,r))=\cH^n(K)$. Then $\di$ is monotonic, measure preserving, and maps balls to balls.  The associated contraction $\psi_{\di,r}:\R^n\to\R^n$ from \eqref{eq:psi1} is defined by $\psi_{\di,r}(x)=ru$, $x\in \R^n$, and hence depends on $r$.} \qed
\end{ex}

The next lemma gives some conditions guaranteeing that $\psi_{\di,r}$ is independent of $r>0$.

\begin{lem}\label{lem11_inependent}
Let ${\mathcal{K}}^n_n\subset{\mathcal{E}}\subset {\mathcal{L}}^n$, let $\di:\cE\to\cL^n$ be monotonic, measure preserving, and map balls to balls, and let $\psi_{\di,r}:\R^n\to\R^n$ be the associated contraction from \eqref{eq:psi1}.

\noindent{\rm{(i)}} If $\di$ is smoothing, then $\psi_{\di,r}$  is independent of $r>0$.

\noindent{\rm{(ii)}} Let $H=u^\perp+t_0u$, $u\in S^{n-1}$, $t_0\in\R$, be a hyperplane. If there is an $r_0>0$ such that the restriction of $\psi_{\di,r_0}$ to $H$ is the identity on $H$, then $\di\big|_{\cK_n^n}$ respects $H$-cylinders and hence $\psi_{\di,r}$ is independent of $r>0$.
\end{lem}

\begin{proof}
(i) Suppose that $\di$ is monotonic, measure preserving, and smoothing.  (The assumption that $\di$ maps balls to balls is superfluous here, since it is guaranteed by \cite[Lemma~4.1]{BGGK2}.) If  $0<r<d$, the smoothing property and \eqref{eq:psi1} yield
$$
\mathrm{int} B(\psi_{\di,r}(x),d)=(\di^* B(x,r))+(d-r)B^n\subset \di^* B(x,d)=\mathrm{int} B(\psi_{\di,d}(x),d)
$$
and hence $\psi_{\di,r}(x)=\psi_{\di,d}(x)$ for all $x\in \R^n$.  Consequently, $\psi_{\di,r}$ does not depend on $r$.

(ii) Assume that $\psi_{\di,r_0}$ acts as the identity on $H$ for some $r_0>0$. We claim that \eqref{eq:psi1} holds with  $\psi_{\di,r}=\psi_{\di,r_0}$ for all $x\in H$ and $r>r_0$.
Indeed, for $r=r_0+\varepsilon$, $\varepsilon>0$, the ball $B(x,r)$, $x\in H$, contains the set
$$D=\bigcup_{z\in x+(
\varepsilon S^{n-1}\cap H)} B(z,r_0),$$
so monotonicity implies that
the ball $\di B(x,r)= B(\psi_{\di,r}(x),r)$ contains
\[
\bigcup_{z\in x+(\varepsilon S^{n-1}\cap H)}\di B(z,r_0)=D.
\]
It follows that $\psi_{\di,r}(x)=x=\psi_{\di,r_0}(x)$.

Let $K\in \cK_n^n$, and for $x\in H$, let $r_{K,x}=\min\{r\ge r_0: K\subset B(x,r)\}$.  Note that if $y\not\in (K|u^\perp)+\myspan\{u\}$, we can choose an $r\ge r_0$ and an $x\in H$ such that $y\not\in B(x,r)\supset K$ and hence $y\not\in B(x,r_{K,x})$.  Therefore
$$K\subset L= \bigcap_{x\in H} B(x,r_{K,x})\subset(K|u^\perp)+\myspan\{u\},$$
where $L\in \cK_n^n$, and monotonicity shows that $\di$ respects $H$-cylinders.  Then Lemma~\ref{lem11}(ii) implies that $\psi_{\di,r}$ is independent of $r>0$.
\end{proof}

When $\psi_{\di,r}$ is independent of $r>0$, we shall henceforth write $\psi_{\di}$ instead. In the following example, we give explicit formulas for the contractions associated with the set maps from Example~\ref{ex429}.  In the formulas, we use the first equation in \eqref{eq:omregn} and the fact that $x=x|u^{\perp}+\langle x,u\rangle u$ for $x\in\R^n$ and $u\in S^{n-1}$.

\begin{ex}\label{ex310}
{\em Let $H=u^{\perp}+t_0u$, $u\in S^{n-1}$, $t_0\in \R$.

(i) The reflection $R_{H}$ in $H$ has associated contractions
$$
\varphi_{R_H}(t)=2t_0-t,~~t\in\R~~\quad{\text{and}}~~\quad
\psi_{R_H}(x)=x+2(t_0-\langle x,u\rangle)u,~~x\in\R^n.
$$

(ii) The polarization $P_H$ with respect to $H$, oriented positively in the direction $u$, has associated contractions
\begin{equation}\label{f2}
\varphi_{P_H}(t)=|t-t_0|+t_0,~~t\in\R~~\quad{\text{and}}~~\quad
\psi_{P_H}(x)=x+(|t_0-\langle x,u\rangle|+(t_0-\langle x,u\rangle))u,~~x\in\R^n.
\end{equation}
Note that $R_H\circ P_H$ is the polarization with respect to $H$ with opposite orientation.

(iii) The Steiner symmetrization $S_H$ with respect to $H$ has associated contractions
$$
\varphi_{S_H}(t)=t_0,~~t\in\R~~\quad{\text{and}}~~\quad
\psi_{S_H}(x)=x+(t_0-\langle x,u\rangle)u,~~x\in\R^n.
$$

(iv) The Solynin set map with respect to $H$, oriented positively in the direction $u$, has associated contractions
\begin{equation}\label{f4}
\varphi_{So_H}(t)=\max\{t_0,t\},~~t\in\R~~\quad{\text{and}}~~\quad
\psi_{So_H}(x)=x+\max\{0,t_0- \langle x,u\rangle\}u,~~x\in\R^n.
\end{equation}

(v) The Brock set map with respect to $H$ and parameter $0\le b\le 1$, oriented positively in the direction $u$, has associated contractions
\begin{align*}
\varphi_{B_H}(t)&=b(t-t_0)+t_0,~~t\in\R~~\quad{\text{and}}~~\quad
\psi_{B_H}(x)=x+(b-1)(\langle x,u\rangle-t_0)u,~~x\in\R^n. ~~~~\quad\quad\qed
\end{align*}
}
\end{ex}

Recall the following consequence of \cite[Theorem~6.6]{BGGK1}.

\begin{thm}\label{thmm7}
Let $n\ge 2$, let $H=u^{\perp}+t_0u$, $u\in S^{n-1}$, $t_0\in \R$,
let $\di:{\mathcal{K}}^n_n \rightarrow {\mathcal{L}}^n$ be a monotonic and measure-preserving map that respects $H$-cylinders and maps balls to balls, and let $\varphi_{\di}$ be its associated contraction from \eqref{ball}. Then for each $K\in \cK^n_n$ and ${\mathcal{H}}^{n-1}$-almost all $x\in H$,
\begin{equation}\label{kseg}
(\di K)\cap (H^{\perp}+x)=(K\cap (H^{\perp}+x))+(\varphi_{\di}(t_x)-t_x) u,
\end{equation}
up to a set of ${\mathcal{H}}^1$-measure zero, where $x+t_xu$ is the midpoint of $K\cap (H^{\perp}+x)$.
\end{thm}

We do not have such a result for maps $\di$ that do not necessarily respect $H$-cylinders; see Problem~\ref{prob2}.

A {\em body} in $\R^n$ is a regular compact set, i.e., a compact set equal to the closure of its interior.  For later use, we record the following result on the convergence of contractions.

\begin{lem}\label{lem11824}
Let $C$ be a body in $\R^n$, $n\ge 1$, and let $\psi_k:\R^n\to \R^n$, $k\in \N$, be contractions. The following are equivalent.

\noindent{\rm(i)} $(\psi_k)$ converges pointwise on $C$.

\noindent{\rm(ii)} $(\psi_k)$ converges uniformly on $C$.

\noindent{\rm(iii)} $(\psi_k)$ converges in $L^1$ on $C$.
\end{lem}

\begin{proof}
Suppose that (i) holds, i.e., $(\psi_k)$ converges pointwise to a function $\psi$ on $C$. It is easy to see that $\psi$ is a contraction on $C$.  Let $r>0$ and $\varepsilon>0$ be given and let $x_1,\ldots,x_m\in C$ be such that the open balls $\inte B(x_i,r)$, $i=1,\dots,m$, cover $C$. Let $x\in C$ and choose $i\in \{1,\dots,m\}$ such that $x\in B(x_i,r)$.  Since $\psi$ and $\psi_{k}$ are contractions, we have $\|\psi(x)-\psi(x_i)\|\le\ee$ and $\|\psi_{k}(x)-\psi_{k}(x_i)\|\le\ee$ for all $k\in \N$. Then
\begin{align*}
\|\psi_k(x)-\psi(x)\|\le \|\psi_k(x)-\psi_k(x_i)\|+\|\psi_k(x_i)-\psi(x_i)\|+\|\psi(x_i)-\psi(x)\|\le 3 \varepsilon,
\end{align*}
for all $k\ge k_0$, where $k_0$ can be chosen independently of $i$ and $x$. This implies (ii).

If (ii) holds, then
$$\|\psi_k-\psi\|_{C,1}=\int_C\|\psi_k(x)-\psi(x)\|\,dx\le \|\psi_k-\psi\|_{C,\infty}\,{\cH}^n(C)\to 0$$
as $k\to\infty$, proving (iii).

Finally, suppose that (iii) holds.  If (i) is false, there is an $x\in C$ and an $\varepsilon>0$ such that $\|\psi_{k_j}(x)-\psi(x)\|>\ee$ for all $j\in \N$.  It is well known (see, for example, \cite[Theorem~2.7]{LL}) that since $\|\psi_{k_j}-\psi\|_{C,1}\to 0$ as $j\to \infty$, there is a subsequence $(\psi_{k_{j_m}})$ of $(\psi_{k_{j}})$ that converges to $\psi$ almost everywhere on $C$.  Since $C$ is a body, we can choose $y\in C$ with $\|x-y\|\le \ee/3$ such that $\psi_{k_{j_m}}(y)\to \psi(y)$ as $m\to\infty$.  Since $\psi$ and $\psi_{k_{j_m}}$ are contractions, we have $\|\psi(x)-\psi(y)\|\le\ee/3$ and $\|\psi_{k_{j_m}}(x)-\psi_{k_{j_m}}(y)\|\le\ee/3$ for all $m\in \N$.  Then
\begin{align*}
\|\psi_{k_{j_m}}(x)-\psi(x)\|\le \|\psi_{k_{j_m}}(x)-\psi_{k_{j_m}}(y)\|+\|\psi_{k_{j_m}}(y)-\psi(y)\|+
\|\psi(y)-\psi(x)\|<\ee
\end{align*}
for sufficiently large $m$, a contradiction.
\end{proof}

\section{Approximation of rearrangements and set maps}\label{Approxrearrange}

In this section, we consider mappings $T: X\to  {\cV}(\R^n)$, where $X\subset {\cV}(\R^n)$.  Let $1\le p< \infty$.  A sequence of such maps $T_k$, $k\in \N$, {\em converges to $T: X\to {\cV}(\R^n)$ in $L^p$} on $X$ if
\[
\|T_kf-Tf\|_p\to 0
\]
as $k\to\infty$ for all $f\in X$. Let $\cU$ and $\cW$ be two families of mappings from $X$ to ${\cV}(\R^n)$.  We say that a map $T\in\cU$ is {\em approximable in $L^p$} on $X$ by maps in $\cW$ if there are $T_k\in \cW$, $k\in \N$, such that $T_k\to T$ in $L^p$ on $X$ as $k\to \infty$, and {\em weakly approximable in $L^p$} on $X$ by maps in $\cW$ if for all $f\in X$, there are $T_{f,k}\in \cW$, $k\in \N$, such that
$$
\|T_{f,k}f-Tf\|_p\to 0
$$
as $k\to \infty$.  We also call a map $T\in\cU$ {\em sequentially approximable in $L^p$} on $X$ by maps in $\cW$ if there is a sequence $(T_k)$, $k\in \N$, of maps in $\cW$ such that
$$
T_k\circ T_{k-1}\circ \cdots\circ T_1\longrightarrow T
$$
in $L^p$ on $X$ as $k\to \infty$, and {\em weakly sequentially approximable in $L^p$} on $X$ by maps in $\cW$ if for all $f\in X$, there is a sequence $(T_{f,k})$, $k\in \N$, of maps in $\cW$ such that
$$
\|(T_{f,k}\circ T_{f,k-1}\circ \cdots\circ T_{f,1})f-Tf\|_p \to 0
$$
as $k\to \infty$.

Suppose that $X\subset {\cV}(\R^n)$ and $\cE=\{A\in \cL^n:1_A\in X\}$. Then the induced map $\di_T:\cE\to {\cL}^n$ given by \eqref{eqIndic} is well defined, and all the previous definitions transfer from maps $T:X\to {\cV}(\R^n)$ to their associated maps $\di_T:\cE\to {\cL}^n$.  In order to allow discussion of the approximation of set mappings independently, however, we record the following parallel definitions.

Consider mappings $\di: \cE\to\cL^n$, where $\cE\subset \cL^n$. A sequence of such maps $\di_k$, $k\in \N$, {\em converges to $\di: \cE\to\cL^n$ in $L^p, 1\le p< \infty$}, if
\[
\|1_{\di_k A}-1_{\di A}\|_p\to 0
\]
as $k\to\infty$ for all $A\in \cE$. Since the left-hand side is the same for all $1\le p<\infty$, it is only necessary to consider $p=1$, so we shall say that $\di_k$ {\em converges to $\di$} and write $\di_k\to \di$ if $\di_k$ converges to $\di$ in $L^1$.

Let $\cF$ and $\cG$ be two families of mappings from $\cE$ to $\cL^n$.  We say that a map $\di\in\cF$ is {\em approximable} on $\cE$ by maps in $\cG$ if there are $\di_k\in \cG$, $k\in \N$, such that $\di_k\to \di$ as $k\to \infty$, and {\em weakly approximable} on $\cE$ by maps in $\cG$ if for all $A\in \cE$, there are $\di_{A,k}\in \cG$, $k\in \N$, such that
$$
\|1_{\di_{A,k}A}-1_{\di A}\|_1\to 0
$$
as $k\to \infty$.  We also call a map $\di\in\cF$ {\em sequentially approximable} on $\cE$ by maps in $\cG$ if there is a sequence $(\di_k)$, $k\in \N$, of maps in $\cG$ such that
$$
\di_k\circ\di_{k-1}\circ \cdots\circ\di_1\longrightarrow \di
$$
as $k\to \infty$, and {\em weakly sequentially approximable} on $\cE$ by maps in $\cG$ if for all $A\in \cE$, there is a sequence $(\di_{A,k})$, $k\in \N$, of maps in $\cG$ such that
$$
\|1_{(\di_{A,k}\circ\di_{A,k-1}\circ \cdots\circ\di_{A,1})A}-1_{\di A}\|_1 \to 0
$$
as $k\to \infty$.

If $\cG$ is closed under finite compositions, then (weak) sequential approximability implies (weak) approximability.

For $1\le p<\infty$ and $X\subset L^p_+(\R^n)$, let $\cl_{\!p}\,X$ be the closure of $X$ in the $L^p$ norm.

\begin{lem}\label{lemLpclosure}
Let $1\le p<\infty$, let $X\subset L^p_+(\R^n)$, let $T:\cV(\R^n)\to \cV(\R^n)$ be a rearrangement, and let $\cW$ be a family of rearrangements from $\cV(\R^n)$ to $\cV(\R^n)$.  The following statements are equivalent.

\noindent{\rm{(i)}} $T$ is weakly approximable in $L^p$ on $X$ by maps in $\cW$.

\noindent{\rm{(ii)}} $T$ is weakly approximable in $L^p$ on $\cl_{\!p} X$ by maps in $\cW$.

This equivalence remains true if ``weakly approximable" is replaced throughout by ``approximable" or by ``sequentially approximable."
\end{lem}

\begin{proof}
We first consider the case of weak approximation. Since (ii) clearly implies (i), we only have to show the converse. To do so, assume that (i) holds and let $f\in \cl_{\!p} X$. For any $\ee>0$, there is a function $g\in X$ with $\|f-g\|_p\le \ee/3$. Due to (i) there is a rearrangement $T_g\in \cW$  with $\|T_gg-Tg\|_p\le \ee/3$. Using the $L^p$-contraction property of $T$ and $T_g$, we obtain
\begin{align*}
\| T_gf-Tf\|_p&\le \| T_gf-T_gg\|_p+\| T_gg-Tg\|_p+\| Tg-Tf\|_p\\
& \le  2\|f-g\|_p+\| T_gg-Tg\|_p\le \ee,
\end{align*}
as required.

Suppose instead that $T$ is approximable in $L^p$ on $X$ by maps in $\cW$.  Then there is a sequence $(T_k)$ in $\cW$ that converges to $T$ in $L^p$ on $X$. The rearrangement $T_g$ can be chosen to be a member of this sequence, and the proof above shows that $T$ is approximable in $L^p$ on $\cl_{\!p}X$ by maps in $(T_k)$.

If we assume instead that $T$ is sequentially approximable in $L^p$ on $X$ by maps in $\cW$, then $(T_k)$ satisfies
$$T_{k}=S_{k}\circ S_{k-1}\circ\cdots S_{1},$$
where $(S_{k})$ is a sequence of maps in $\cW$, but otherwise the proof is the same.
\end{proof}

In this context, see Problem~\ref{probTranslation}.

\begin{thm}\label{thm514}
Let $1\le p<\infty$.  Suppose that $T:{\cV}(\R^n)\to {\cV}(\R^n)$ is a rearrangement, and let $\cW$ be a family of rearrangements from ${\cV}(\R^n)$ to $\cV(\R^n)$. The following statements are equivalent.

\noindent{\rm{(i)}} $T$ is approximable in $L^p$ on $L^p_+(\R^n)$ by maps in $\cW$.

\noindent{\rm{(ii)}} $\di_T$ is approximable on $\mathcal L^n$ by maps in $\{\di_W:W\in \cW\}$.

\noindent{\rm{(iii)}} $\di_T$ is approximable on $\mathcal C^n$ by maps in $\{\di_W:W\in \cW\}$.

The equivalence remains true if ``approximable" is replaced throughout by ``sequentially approximable."  Moreover, {\em (i)}$\Rightarrow${\em (ii)}$\Leftrightarrow${\em (iii)} also holds if ``approximable" is replaced throughout by ``weakly approximable."
\end{thm}

\begin{proof}
Suppose that (i) is true.  The characteristic functions of sets in ${\cL}^n$ are contained in $L^p_+(\R^n)$. It follows that there is a sequence $(W_{k})$ of maps in ${\cW}$ such for each $A\in {\cL}^n$, we have $\|W_{k} 1_A-T1_A\|_p\to 0$ as $k\to \infty$. By \eqref{eqoct72}, this is equivalent to $\|1_{\di_{W_k}A}-1_{\di_{T}A}\|_p\to 0$ as $k\to \infty$, where we make take $p=1$.  This proves (ii). Obviously (ii) implies (iii).

Suppose that (iii) holds.  Using \eqref{eqoct72} and the above definitions of approximability for rearrangements and set maps, we see that $T$ is approximable in $L^1$ on $\{1_C:C\in {\mathcal{C}}^n\}$ by maps in $\cW$. By Lemma~\ref{lemLpclosure}, this also holds on $\cl_{\!1}\{1_C:C\in \mathcal C^n\}$.  Since each set in $\cL^n$ can be approximated in $L^1$ by compact sets, (ii) follows.

The arguments above are easily adapted when ``approximable" is replaced throughout by ``weakly approximable" or ``sequentially approximable."

Finally, assume (ii).  Let $X$ be the class of nonnegative simple integrable functions, and let $f=\sum_{i=1}^m \alpha_i 1_{A_i}\in X$ be a simple function with $0<\alpha_m<\alpha_{m-1}<\cdots<\alpha_1$ and disjoint sets $A_1,\ldots,A_m\in \cL^n$.  If $S$ is any rearrangement on $X$, then using \eqref{eqoct62} with $T$ replaced by $S$, we obtain
\[
(Sf)(x)=\max\{0,\sup\{\alpha_k: x\in \di_S(A_1\cup\cdots\cup A_k)\}\},
\]
essentially. Setting $B_k=A_1\cup\cdots\cup A_k$, $k=1,\ldots,m$, and $B_0=\emptyset$, we have
\begin{equation}
\label{eq:Rearrange_simple}
Sf=\sum_{k=1}^m\alpha_k
1_{(\di_S B_k)\setminus \di_S B_{k-1}},
\end{equation}
essentially.

Let $\ee>0$ be given. Since $B_1,\ldots,B_m\in  \cL^n$, (ii) guarantees the existence of $W\in \cW$ such that
\[
\|1_{\di_W B_k}-1_{\di_T B_k}\|_1\le \frac{\ee}{4\sum_{i=1}^m\alpha_i}
\]
for $k=1,\ldots,m$. By \eqref{eq:Rearrange_simple}, applied with $S$ replaced by $T$ and by $W$, we have
\[
\|Tf-Wf\|_p\le
\sum_{k=1}^m\alpha_k \|1_{(\di_W B_k)\setminus \di_W B_{k-1}}-1_{(\di_T B_k)\setminus\di_T B_{k-1}}\|_1\le\ee.
\]
We conclude that $T$ is approximable in $L^p$ on $X$ by maps in $\cW$. Since $\cl_{\!p}\,X=L^p_+(\R^n)$, (i) now follows from Lemma~\ref{lemLpclosure}.

The proof of (ii)$\Rightarrow$(i) is easily adapted when ``approximable" is replaced by ``sequentially approximable."
\end{proof}

The main interest here is in approximation by polarizations and we begin the discussion with a couple of simple remarks.  Firstly, (weak) sequential approximability by polarizations is the same as (weak) sequential approximability by finite compositions of polarizations, so we only use the former simpler terms.  However, with the other forms of approximation, the finite compositions of polarizations, and not just the polarizations, is the appropriate class to consider.  For example, if $\cE\subset{\mathcal{B}}^n$ and $\di:\cE\to {\cL}^n$ is weakly approximable by polarizations, then for each $A\in{\cE}$, either $\di A=A$, essentially, or $\di A=P_{H_A}A$, essentially, for some hyperplane $H_A$ depending on $A$.  We omit the easy proof. It follows that if $\di$ is approximable by polarizations, then $\di$ is essentially either the identity or a polarization itself.

Van Schaftingen \cite[Theorem~1 and Section~4.3]{VS2} proved that Steiner and Schwarz rearrangements are sequentially approximable in $L^p$ on $L^p_+(\R^n)$ by polarizations. By Theorem~\ref{thm514}, the same is true for the associated set maps on ${\cL}^n$. In contrast, Theorem~\ref{thm514} and the following lemma show that Brock and Solynin rearrangements are not sequentially approximable in $L^p$ on $L^p_+(\R^n)$ by polarizations.

\begin{lem}\label{Markuslem}
Let $\cE\subset\cL^n$ contain all balls and let $\di:\cE\to \cL^n$ be  sequentially approximable on $\cE$ by polarizations. Then there is a hyperplane $H$ such that for each ball $B$ in $\R^n$, we have $\di B=\di B^{\dagger}$, where $B^{\dagger}$ is the reflection of $B$ in $H$.

It follows that translations, Brock set maps with $b\ne 0$, and Solynin set maps (both with respect to arbitrary hyperplanes) are not sequentially approximable on $\cE$ by polarizations.
\end{lem}

\begin{proof}
Let $(\di_k)$ be a sequence of polarizations such that $\di_k\circ\dots\circ \di_1\to \di$ as $k\to\infty$. Suppose that $\di_1$ is taken with respect to an oriented hyperplane $H$.  From \eqref{pol2} with $P_H=\di_1$, we see that if $B$ is any ball, then $\di_1 B=B$ if the center of $B$ lies in $H^+$, while $\di_1 B=B^{\dagger}$ if the center of $B$ lies in $H^-$.  Therefore
$\di_1 B=\di_1 B^{\dagger}$. Applying $\di_k\circ\dots\circ \di_2$ to both sides and taking the limit as $k\to\infty$, we obtain $\di B=\di B^{\dagger}$.

Since translations and Brock set maps with $b\ne 0$ are injective on the family of balls in $\R^n$, it follows immediately that they are not sequentially approximable on $\cE$ by polarizations.

Suppose that the Solynin set map $\di_{So_{H_0}}$ with respect to a hyperplane $H_0$ is sequentially approximable on $\cE$ by polarizations, and let $H$ be as in the first part of the lemma.  Assume first that $H\cap \inte H_0^+\neq\emptyset$. Clearly there is a ball $B\subset H_0^+\cap H^+$ such that the reflection $B^{\dagger}$ of $B$ in $H$ is contained in $H_0^+\cap H^-$.  Then $\di_{So_{H_0}}B=B\neq B^{\dagger}=\di_{So_{H_0}} B^{\dagger}$, a contradiction. Consequently, $H$ must be parallel to $H_0$ and contained in $H_0^-$. But if $B\subset H_0^+$ is a ball, then $\di_{So_{H_0}}B=B$, while $B^{\dagger}\subset H_0^-$ means (see Example~\ref{ex429}(iv)) that the center of $\di_{So_{H_0}}B^{\dagger}$ is in $H_0$.  Therefore $\di_{So_{H_0}}B\neq \di_{So_{H_0}} B^{\dagger}$, another contradiction.
\end{proof}

However, the translation $\di A=A+au$, $u\in S^{n-1}$, $a\ge 0$, is weakly sequentially approximable on ${\mathcal B}^n$ by polarizations.  To see this, let $P_{t,u}$ be the polarization with positive halfspace $\{z\in \R^n:\langle z,u\rangle \ge t\}$, bounded by the hyperplane $H_{t,u}=u^{\perp}+tu$.  Suppose that $\{x,x+(a/2)u\}\subset H^-_{t,u}$. Then $P_{{t,u}}x=x^{\dagger} =x-2(\langle x,u\rangle-t)u$, the reflection of $x$ in $H_{t,u}$. Therefore $P_{{-t-a/2,-u}}x^{\dagger}=x+au$, the reflection of $x^{\dagger}$ in $H_{t+a/2,u}$. If $A\in {\mathcal B}^n$, we can choose $t\in \R$ such that $A\cup (A+(a/2)u)\subset H^-_{t,u}$ and conclude that $\di_t=P_{-t-a/2,-u}\circ P_{t,u}$ satisfies $\di_t A=A+au$. We do not know if translations are weakly sequentially approximable on $\cL^n$; see Problem~\ref{probTranslation}.

The translation $\di A=A+au$, $u\in S^{n-1}$, $a\ge 0$, is approximable on ${\cL}^n$ by finite compositions of polarizations.  The sequence $(\di_k)$, with $\di_t$ defined as in the previous paragraph, is an approximating sequence. Indeed, if $A\in {\mathcal B}^n$, and $k_0\in \N$ is chosen such that $A\cup (A+(a/2)u)\subset H_{k_0,u}$, then $\di_kA=A+au$ for all $k\ge k_0$, as we saw above.  By Theorem~\ref{thm514}, this remains true with ${\mathcal B}^n$ replaced by $\cL^n$.

Solynin \cite{Sol1} proved that the Solynin rearrangement $So_H$ is weakly approximable in $L^p$ on $L^p_+(\R^n)$ by finite compositions of polarizations. By Theorem~\ref{thm514}, the same is true for the associated set map on ${\cL}^n$.  In Theorem~\ref{gabcor} below, we show that ``weakly approximable" can be replaced by ``approximable" in these results.  We do not know if it can be replaced by ``weakly sequentially approximable"; see Problem~\ref{prob6}.

We do not know if a map $\di:\cE\to \cL^n$ which is approximable by finite compositions of polarizations is also weakly sequentially approximable by polarizations, or if the converse is true; see Problem~\ref{prob3}.

\section{Approximation by polarizations with respect to parallel hyperplanes}\label{Approximation}

Recall that if $\cE\subset \cL^n$, then $\cJ(\cE)$ denotes the family of monotonic and measure-preserving maps from $\cE$ to $\cL^n$ that map balls to balls, and if $H$ is a hyperplane in $\R^n$, then $\cJ_H(\cE)\subset \cJ(\cE)$ is the subfamily that respect $H$-cylinders.

\begin{lem}\label{lem12024}
Let $n\ge 2$, let $H=u^{\perp}+t_0u$, $u\in S^{n-1}$, $t_0\in \R$, let $\di_k\in \cJ_H(\cK_n^n)$, $k=0,1,\dots$, and let $\varphi_{\di_k}:\R\to\R$, $k=0,1,\dots$, be their associated contractions from \eqref{ball}.

\noindent{\rm{(i)}} If $\di_k\to \di_0$ as $k\to\infty$, then $\varphi_{\di_k}\to \varphi_{\di_0}$ pointwise as $k\to \infty$.

\noindent{\rm{(ii)}} If $\varphi_{\di_k}\to \varphi$ pointwise as $k\to\infty$, there is an essentially unique $\di\in \cJ_H(\cK_n^n)$, with associated contraction $\varphi$ from \eqref{ball}, such that $\di_k\to\di$ as $k\to\infty$.
\end{lem}

\begin{proof}
To show (i), let $t\in \R$ and use Lemma~\ref{lemm5} and the $L^1$-convergence of $\di_k$ to $\di$ to conclude that
\[
B(\varphi_{\di_k}(t)u,1)=\di_k B(tu,1)\to\di B(tu,1)=B(\varphi_\di(t)u,1),
\]
as $k\to\infty$, implying that $\varphi_{\di_k}\to \varphi_\di$ pointwise as $k\to\infty$.

To prove (ii), assume that $\varphi_{\di_k}\to \varphi$ pointwise as $k\to\infty$. In view of Lemma~\ref{lem11824}, this convergence holds uniformly on bounded intervals. It is easy to see that $\varphi$ is a contraction.  The map $\di$ defined via \eqref{kseg} has all the required properties apart from the claimed convergence.

To prove this, let $K\in \cK_n^n$, let $\varepsilon>0$, and let $I\subset \R$ be an interval such that $\{t_x: x\in K|H\}\subset I$, where $t_x$ is the midpoint of $K\cap (H^{\perp}+x)$ for each $x\in H$. The uniform convergence of $\varphi_{\di_k}$ to $\varphi$ on $I$ yields a $k_0$ such that $|\varphi_{\di_k}(t)-\varphi(t)|\le \varepsilon$ for all $k\ge k_0$ and $t\in I$. It is evident from \eqref{kseg} applied to $\di$ and to $\di_k$, $k\ge k_0$, that
\[
 \di_k K\subset  (\di K)+\varepsilon [-u,u]~~\quad{\text{and}}~~\quad \di K\subset  (\di_k K)+\varepsilon [-u,u]
\]
essentially, for $k\ge k_0$. These inclusions and Fubini's theorem imply that $\di_kK\to \di K$ as $k\to\infty$.
\end{proof}

\begin{lem}\label{lem311}
Let $H=e_n^{\perp}$, let $s_1<s_2$, and let $K$ be any convex body in $\R^n$ such that
\begin{equation}\label{eq51}
[s_1,s_2]\subset\{t_x: x\in K|H\},
\end{equation}
where $x+t_xe_n$ is the midpoint of the line segment $K\cap (H^{\perp}+x)$.  Let $\di\in \cJ_H(\cK_n^n)$ and suppose that there are $\di_{K,k}\in \cJ_H(\cK_n^n)$, $k\in \N$, such that
\begin{equation}\label{eq3111}
\|1_{\di_{K,k}K}(x)-1_{\di K}(x)\|_1\to 0
\end{equation}
as $k\to \infty$.  Then the associated contractions $\varphi_{\di_{K,k}}$ from \eqref{ball} converge uniformly to $\varphi_{\di}$ on $[s_1,s_2]$.
\end{lem}

\begin{proof}
Suppose that $K\in\cK^n_n$ satisfies \eqref{eq51} and that there are $\di_{K,k}\in \cJ_H(\cK_n^n)$, $k\in \N$, satisfying \eqref{eq3111}. Then
$$\int_{\R^n}|1_{\di_{K,k}K}(x)-1_{\di K}(x)|\,dx\to 0$$
and hence $\cH^n\left((\di_{K,k}K)\,\triangle\, \di K\right)\to 0$ as $k\to\infty$.  By Fubini's theorem,
\begin{equation}\label{eq54}
\cH^1\left(((\di_{K,k}K)\cap(H^{\perp}+x))\,\triangle\,((\di K)
\cap(H^{\perp}+x))\right)\to 0,
\end{equation}
as $k\to\infty$, for $\cH^{n-1}$-almost all $x\in K|H$. By \eqref{kseg}, we have
\begin{equation}\label{eq52}
(\di_{K,k}K)\cap(H^{\perp}+x)=K\cap(H^{\perp}+x)+(\varphi_{\di_{K,k}}(t_x)
-t_x)e_n
\end{equation}
for $k\in \N$ and
\begin{equation}\label{eq53}
(\di K)\cap(H^{\perp}+x)=K\cap(H^{\perp}+x)+(\varphi_{\di}(t_x)
-t_x)e_n,
\end{equation}
up to sets of $\cH^1$-measure zero, for $\cH^{n-1}$-almost all $x\in K|H$. From \eqref{eq51}, \eqref{eq54}, \eqref{eq52}, and \eqref{eq53}, we obtain $\varphi_{\di_{K,k}}(s)\to \varphi_{\di}(s)$ for $\cH^{1}$-almost all $s\in [s_1,s_2]$. Since $\varphi_{\di_{K,k}}$, $k\in \N$, and $\varphi_{\di}$ are contractions, the convergence holds for all $s\in [s_1,s_2]$ and hence, by Lemma~\ref{lem11824}, we have
$\varphi_{\di_{K,k}}\to \varphi_{\di}$ uniformly on $[s_1,s_2]$ as $k\to \infty$.
\end{proof}

An example of a convex body satisfying \eqref{eq51} (with equality) is the parallelepiped
$$K=[-e_2,e_2]+\cdots +[-e_n,e_n]+[s_1(e_1+e_n), s_2(e_1+e_n)],$$
for which $K|H=[-e_2,e_2]+\cdots +[-e_{n-1},e_{n-1}]+[s_1e_1, s_2e_1]$.

\begin{lem}\label{lem1}
Let $H$ be a hyperplane in $\R^n$ and let $\cJ'\subset \cJ_H(\cK_n^n)$.  A map $\di\in \cJ_H(\cK_n^n)$ is approximable on ${\cK}^n_n$ by maps in $\cJ'$ if and only if it is weakly approximable on ${\cK}^n_n$ by maps in $\cJ'$.
\end{lem}

\begin{proof}
Suppose that $\di\in \cJ_H(\cK_n^n)$ is weakly approximable on ${\cK}^n_n$ by maps in $\cJ'$. Let $K_m$, $m\in \N$, satisfy the hypotheses of Lemma~\ref{lem311} with $s_1=-m$ and $s_2=m$. Then there are $\di_{K_m,k}\in \cJ'$, $k\in \N$, such that \eqref{eq3111} holds with $K$ replaced by $K_m$.  By Lemma~\ref{lem311}, $\varphi_{K_m,k}\to \varphi$ uniformly on $[m,m]$ as $k\to\infty$. Hence, there is a $k_m\in\N$ such that $|\varphi_{K_m,k}(x)- \varphi(x)|\le 1/m$ for all $x\in [-m,m]$ and  $k\ge k_m$.  It follows that the maps $\varphi_m=\varphi_{K_m,k_m}\to\varphi$ pointwise on $\R$ as $m\to\infty$. Lemma~\ref{lem12024}(ii) now implies that $\di_m=\di_{K_m,k_m}\to \di$ as $m\to\infty$, so $\di$ is approximable on ${\cK}^n_n$ by maps in $\cJ'$.  The converse is obvious.
\end{proof}

We shall write $I(a,b)$ for the open interval in $\R$ with endpoints $a$ and $b$, where $a\ne b$, and  $I(a,a)=\emptyset$. Let $\cI$ be the class of contractions $\varphi:\R\to \R$ such that whenever $a<b$ and $\varphi((a,b))\subset I(\varphi(a),\varphi(b))$, the restriction $\varphi|_{(a,b)}$ is an affine function on $(a,b)$ with slope $\pm 1$. More formally, $\varphi\in\cI$ if whenever $a<b$ and
\begin{equation}\label{I1m}
	\varphi(s)\in I(\varphi(a),\varphi(b))
	\ \ \ \forall s\in (a,b),
\end{equation}
there is a $\delta\in \{-1,1\}$ such that
\[
\varphi(s)=\delta(s-a)+\varphi(a)\ \ \ \forall s\in (a,b).
\]
Since $\varphi$ is a contraction, the existence of $\delta\in \{-1,1\}$ such that the last displayed formula holds is equivalent to
$$
|\varphi(b)-\varphi(a)| = b-a.
$$

A contraction $\varphi$ may be piecewise linear, with each piece having slope $\pm 1$, and yet not belong to $\cI$. For example, suppose that $\varphi(t)=t$, if $t<0$, $\varphi(t)=-t$, if $0\le t\le 1$, and $\varphi(t)=t-2$, if $t>1$.  Then
$$-2=\varphi(-2)<\varphi(s)<\varphi(3)=1$$
for $-2<s<3$, but $\varphi$ is not affine on $(-2,3)$.

A class of real-valued functions on $\R$ is called \emph{bi-reflection invariant} if it contains $-\varphi$ and $s\mapsto \varphi(-s)$ whenever $\varphi$  is in this class. It is called
\emph{bi-translation invariant} if it contains $\varphi(\cdot+s_0)$ and $\varphi(\cdot)+s_0$ for all $s_0\in \R$ whenever $\varphi$ is in this class.

\begin{lem}\label{restate}
The family $\cI$ has the following properties.

\noindent{\rm{(i)}} ${\rm{Id}}\in \cI$, $(t\mapsto |t|)\in \cI$.

\noindent{\rm{(ii)}} $\cI$ is bi-reflection invariant.

\noindent{\rm{(iii)}} $\cI$ is bi-translation invariant.

\noindent{\rm{(iv)}} $\cI$ is closed under composition.

\noindent{\rm{(v)}} $\cI$ is closed with respect to pointwise convergence.
\end{lem}

\begin{proof}
(i), (ii) and (iii) are obvious.

To show (iv), let $\varphi_1,\varphi_2\in \cI$ and let $\varphi=\varphi_2\circ \varphi_1$. Suppose that $a<b$ are such that \eqref{I1m} holds. For brevity, we define $\tilde a=\varphi_1(a)$ and $\tilde b=\varphi_1(b)$. If $r\in I(\tilde a, \tilde b)$, then by the intermediate value theorem applied to $\varphi_1$, there is an $s\in (a,b)$ such that $\varphi_1(s)=r$ and hence $\varphi_2(r)=\varphi_2(\varphi_1(s))\in I(\varphi_2(\tilde a),\varphi_2(\tilde b))$.  Since $\varphi_2\in \cI$, this shows that $\varphi_2(s)=\delta s+c$
for all $s$ in the closure of $I(\tilde a, \tilde b)$, where $c\in \R$ and $\delta\in \{-1,1\}$.

Suppose that there is an $s_1\in (a,b)$ such that  $\varphi_1(s_1)\not\in I(\tilde a, \tilde b)$. Then by the intermediate value theorem again, there is an $s_2\in (a,b)$ with $\varphi_1(s_2)\in \{\tilde a, \tilde b\}$. We only consider the case when $\varphi_1(s_2)=\tilde a$, as the case $\varphi_1(s_2)=\tilde b$ is treated in a very similar way. We obtain
\[
\varphi(s_2)= \delta \varphi_1(s_2)+c=\delta \varphi_1(a)+c=\varphi(a)\not\in I(\varphi(a),\varphi(b)),
\]
contradicting \eqref{I1m}. Thus \eqref{I1m} holds with $\varphi$ replaced by $\varphi_1$.  Since $\varphi_1\in \cI$, its restriction to $(a,b)$ must be an affine function with slope $\pm 1$. The composition of two such functions is again an affine function with slope $\pm 1$ and (iv) follows.

To prove (v), let $(\varphi_k)$ be a sequence in $\cI$ with pointwise limit $\varphi$. Clearly $\varphi$ is a contraction. Consider $a<b$ such that \eqref{I1m} holds. We may assume without loss of generality that $\varphi(a)<\varphi(b)$. For $\varepsilon\in (0,(\varphi(b)-\varphi(a))/4)$, let $I_{\ee}=(a_{\ee},b_{\ee})$, where $a<a_\ee<b_\ee<b$ are such that $\phi(a_\ee)=\phi(a)+2\ee$, $\phi(b_\ee)=\phi(b)-2\ee$, and $\phi(s)\in(\phi(a)+2\ee, \phi(b)-2\ee)$ for all $s\in I_\ee$.  By Lemma~\ref{lem11824}, there is a $k_0\in \N$ such that
\[
|\varphi_k(s)-\varphi(s)|\le \varepsilon
\]
for all $k\ge k_0$ and $s\in (a,b)$. Suppose that $k\ge k_0$.  Then for $s\in I_\ee$, we have
\[
\varphi_k(s)-\varphi_k(a)\ge (\varphi(s)-\varepsilon)  -(\varphi(a)+\varepsilon)= \varphi(s)-(\varphi(a)+2\varepsilon)> 0
\]
and similarly $\varphi_k(b)-\varphi_k(s)>0$, so $\varphi_k: I_\ee \to (\varphi_k(a),\varphi_k(b))$. Let
\[
a_k=\max\{s\in [a,a_\ee]:\varphi_k(s)=\varphi_k(a)\}~~\quad{\text{and}}~~\quad
b_k=\min\{s\in [b_\ee,b]:\varphi_k(s)=\varphi_k(b)\}).
\]
Then $a_k\in (a,a_{\ee})$, $b_k\in (b_{\ee},b)$, and
$\varphi_k: (a_k,b_k)\to(\varphi_k(a_k),\varphi_k(b_k))$. Since $\varphi_k\in \cI$, $\varphi_k$ is affine with slope $1$ on $(a_k,b_k)$ and hence also on the smaller interval $I_{\ee}$.
Therefore $\varphi$ is also affine with slope $1$ on $I_\ee$. As $a_\ee\to a$ and $b_\ee\to b$ as $\ee\to 0$, $\varphi$ is affine with slope $1$ on $[a,b]$, so $\varphi\in\cI$.
\end{proof}

Recall the formula for $\varphi_{P_H}$ for a polarization $P_H$, given in Example~\ref{ex310}(ii).  This and parts (i)--(iv) of the previous lemma immediately yield the following result.

\begin{cor}\label{cor313}
Let $H$ be a hyperplane in $\R^n$.  If $\di$ is a finite composition of polarizations with respect to hyperplanes parallel to $H$, then $\varphi_{\di}\in\cI$.
\end{cor}

The family $\cI$ is not closed under taking maxima and minima.  For example, if $f_1(s)=\max\{0,s\}$ and $f_2(s)=-\max\{0,-s\}+1$, then $f_1,f_2\in\mathcal{I}$ but $\max\{f_1,f_2\}\not\in\cI$ since it is constant on $(0,1)$ and strictly increasing on $(-1,0)\cup(1,\infty)$.

\begin{thm}\label{thmgpPrime}
Let $H$ be a hyperplane in $\R^n$ and let $\di\in \cJ_H(\cK_n^n)$.  If $\varphi_{\di}\not\in \cI$, then $\di$ cannot be weakly approximated on ${\cK}^n_n$ by finite compositions of polarizations with respect to hyperplanes parallel to $H$.
\end{thm}

\begin{proof}
If $\di$ can be weakly approximated on ${\cK}^n_n$ by finite compositions of polarizations with respect to hyperplanes parallel to $H$, then by Lemma~\ref{lem1}, there are such maps $\di_k$, $k\in \N$, such that $\di_k\to \di$ as $k\to\infty$.  By Lemma~\ref{lem12024}(i), $\varphi_{\di_k}\to \varphi_{\di}$ pointwise as $k\to \infty$.  The associated contractions $\varphi_{\di_k}$ from \eqref{ball} belong to $\cI$ by Corollary~\ref{cor313}, so by Lemma~\ref{restate}(v), $\varphi_{\di}\in\cI$, a contradiction.
\end{proof}

\begin{lem}\label{lem314}
Let $\varphi:\R\to\R$ be a contraction such that there exists a point where $\varphi'$ exists and is different from 0, 1, and -1.  Then $\varphi\not\in\cI$.
\end{lem}

\begin{proof}
Suppose that $\varphi'(s_0)$ exists and is different from 0, 1, and -1.
Since $\varphi$ is a contraction, we have $-1<\varphi'(s_0)<0$ or $0<\varphi'(s_0)<1$.  We shall assume the latter holds, since the former can be dealt with similarly. Then there is an $a>0$ satisfying
$$a<\min\{\varphi'(s_0), 1-\varphi'(s_0)\}$$
and $r_1<s_0<r_2$ such that on $[r_1,r_2]$, the graph of $\varphi$ lies between the lines $t=f_{\pm}(s)$ through $(s_0, \varphi(s_0))$ with slopes $\varphi'(s_0)\pm a$.  Let
$$
s_1=\sup\{r\in[r_1,r_2] : \varphi(r)=\min\{\varphi(s): r_1\leq s \leq r_2\}\}
$$
and
$$
s_2=\inf\{r\in[r_1,r_2] : \varphi(r)=\max\{\varphi(s): r_1\leq s \leq r_2\}\}.
$$
Then  $\varphi(s_1)<\varphi(s)<\varphi(s_2)$ for all $s\in (s_1,s_2)$. We also have
$$f_{+}(s_1)\le \varphi(s_1)\le f_-(s_1)~~\quad{\text{and}}~~\quad f_{-}(s_2)\le \varphi(s_2)\le f_+(s_2),$$
so
$$0<\varphi'(s_0)-a\le \frac{f_-(s_2)-f_-(s_1)}{s_2-s_1}\le
\frac{\varphi(s_2)-\varphi(s_1)}{s_2-s_1}\le
\frac{f_+(s_2)-f_+(s_1)}{s_2-s_1}=\varphi'(s_0)+a<1.$$
This shows that $\varphi$ is not affine with slope 1 on $(s_1,s_2)$.
\end{proof}

\begin{cor}\label{cor314}
If $H$ is a hyperplane in $\R^n$, the Brock set map $\di_{B_H}$ with parameter $0<b<1$ cannot be weakly approximated on ${\cK}^n_n$ by finite compositions of polarizations with respect to hyperplanes parallel to $H$.
\end{cor}

\begin{proof}
Since $0<\varphi_{\di_{B_H}}'(s)=b<1$, this follows immediately from Theorem~\ref{thmgpPrime} and Lemma~\ref{lem314}.
\end{proof}

\section{Approximation by arbitrary polarizations}\label{arb}

In the previous section, we focused on contractions $\varphi$ associated with maps that respect $H$-cylinders.  Here we consider the contractions $\psi$ introduced in Lemmas~\ref{lem11} and~\ref{lem11_inependent}.

\begin{thm}\label{lem5}
Let $\psi:\R^n\to\R^n$ be a contraction that can be approximated (pointwise) on $B^n$ by finite compositions of contractions associated with polarizations. If $\psi|_{B^n}$ is injective, then ${\cH}^n(\psi(B^n))=\kappa_n={\cH}^n(B^n)$.
\end{thm}

\begin{proof}
Let $P_0$ be a polarization with respect to an oriented hyperplane $H$ and with associated contraction $\psi_0$.  We shall use the fact that if $B=B(x,r)$ is a ball, then by \eqref{f2}, we have $\psi_0(B)=B\cap H^+$ if $x\in H^+$ and $\psi_0(B)=(R_HB)\cap H^+$ if $x\in H^-$, where $R_H$ denotes reflection in $H$.

Since $\psi$ is a contraction, ${\cH}^n(\psi(B^n))\leq\kappa_n$. Suppose that ${\cH}^n(\psi(B^n))<(1-\delta)^{n}\kappa_n$ for some $\delta>0$. We have to show that $\psi$ is not injective on $B^n$. Let
\[
\psi_k=\psi_{k,m_k}\circ\dots\circ\psi_{k,1},
\]
$k\in \N$, be finite compositions of contractions $\psi_{k,i}$ associated with polarizations $P_{k,i}$ taken with respect to oriented hyperplanes $H_{k,i}$, $i=1,\dots,m_k$, with the property that $\psi_k\to \psi$ pointwise as $k\to\infty$.
For each $i=1,\dots,m_k$, let $x_{k,i}=(\psi_{k,i}\circ\dots\circ \psi_{k,1})(o)$ and $x_{k,0}=o$ for each $k$. Note that
\[
(\psi_{k,i}\circ\dots\circ \psi_{k,1})(B^n)\subset B(x_{k,i},1),
\]
as all mappings are contractions.  If $H_{k,i}$, $i=1,\dots,m_k$, does not meet $\inte B(x_{k,i-1},1-\delta)$, then $\psi_{{k,i}}({B(x_{k,i-1},1-\delta)})={B(x_{k,i},1-\delta)}$, regardless of the orientation of $H_{k,i}$.  If this is the case for each $i=1,\dots,m_k$, then ${B(x_{k,m_k},1-\delta)}\subset\psi_k(B^n)$, so ${\cH}^n(\psi_k(B^n))\ge (1-\delta)^{{n}}\kappa_n$. If moreover this occurs for arbitrarily large $k$, then there is a subsequence $(\psi_{k_j})$ of $(\psi_k)$ such that ${B(x_{k_j,m_{k_j}},1-\delta)}\subset\psi_{k_j}(B^n)$ for each $j\in \N$.  This would imply that $\psi(B^n)$ contains a ball of radius $(1-\delta)$ and hence that ${\cH}^n(\psi(B^n))\ge (1-\delta)^{{n}}\kappa_n$, a contradiction.

Thus there is an $N_1$ such that for each $k\ge N_1$, there is an oriented hyperplane $H_{k,i_k}$, where $1\le i_k\le m_k$ and $i_k$ is as small as possible, which meets $\inte {B(x_{k,i_k-1},1-\delta)}$. Assume from now on that $k\ge N_1$. We claim that there are $y_k$, $z_k\in B^n$ (possibly for $k$ in a subsequence) such that
\begin{equation}\label{non_inject_k}
 \psi_k(y_k)=\psi_k(z_k)\quad\text{and}\quad \liminf_{k\to\infty}\|y_k-z_k\|> 0.
\end{equation}

In order to prove this claim, we assume initially that
\begin{equation}\label{oinpositivehalfsp}
o\in H_{k,i}^+\quad\text{for $k\geq N_1$ and $i=1,\dots,i_k$.}
\end{equation}
This implies that $x_{k,i}=o$ and
$(\psi_{k,i}\circ\dots\circ \psi_{k,1})(B^n)\subset B^n$.
For $i=1,\dots,i_k$, let $\beta_{k,i}\le 1$ and let $w_{k,i}\in S^{n-1}$ be such that
\[
H_{k,i}=\{x : \langle x,w_{k,i}\rangle=1-\beta_{k,i}\}.
\]
The definition of $i_k$ implies that $\beta_{k,i}\leq\delta$, $i=1,\dots,i_k-1$, and $\beta_{k,i_k}>\delta$. We may assume, possibly after extracting a subsequence, that $\lim_{k\to\infty}\beta_{k,i_k}=\beta\geq\delta$.

Suppose first that $\beta>\delta$ and define
\[
y_k=(1-\delta)w_{k,i_k}\quad\text{and}\quad z_k=R_{H_{k,i_k}}(y_k).
\]
We have $y_k, z_k\in (1-\delta) B^n$, $z_k\in H_{k,i_k}^+$, and, by changing $N_1$ if necessary, $\|y_k-z_k\|\geq (\beta-\delta)$ for $k\ge N_1$. Since the map $\psi_{k,i}$ equals the identity on $(1-\delta) B^n$, $i=1,\dots,i_k-1$, and $ \psi_{k,i_k}(y_k)=z_k=\psi_{k,i_k}(z_k)$, we conclude that
\begin{equation}\label{non_inject_ik}
(\psi_{k,i_k}\circ\dots\circ \psi_{k,1})(y_k)=(\psi_{k,i_k}\circ\dots\circ \psi_{k,1})(z_k),
\end{equation}
and \eqref{non_inject_k} follows.

Now suppose that $\beta=\delta$. We may assume, by changing $N_1$ if necessary, that $\beta_{k,i_k}\leq 2\delta$ for $k\ge N_1$.
For $\gamma\in(0,1)$ and $v\in S^{n-1}$, let $C(v,\gamma)$ denote the spherical cap
\[
 C(v,\gamma)=B^n\cap\{x\in \R^n : \langle x,v\rangle \geq \gamma\}.
\]
We may assume that $\delta$ is small enough to ensure that $1-4\delta>0$ and that if $v_1, v_2\in S^{n-1}$ and $C(v_1,1-4\delta)\cap C(v_2,1-4\delta)\neq \emptyset$, then $\langle v_1,v_2\rangle\geq 0$.
Define
\[
y_k=w_{k,i_k}\quad\text{and}\quad z_k=(\psi_{k,i_k}\circ\dots\circ\psi_{k,1})(y_k).
\]
To show that \eqref{non_inject_k} holds, we first prove by induction on $i$ that
\begin{equation}\label{intersec_halfsp}
(\psi_{k,i}\circ\dots\circ\psi_{k,1})(y_k)\in H_{k,1}^+\cap \dots \cap H_{k,i}^+
\end{equation}
for $i=1,\dots,i_k$.  Note that for each $i$, $(\psi_{k,i}\circ\dots\circ\psi_{k,1})(y_k)\in H_{k,i}^+$, by the definition of $\psi_{k,i}$. It follows that \eqref{intersec_halfsp} is valid for $i=1$.  Suppose that \eqref{intersec_halfsp} is true for $i=m<i_k$. Let $y_k^{(m)}=(\psi_{k,m}\circ\dots\circ\psi_{k,1})(y_k)$. If $y_k^{(m)}\in H_{k,m+1}^+$, then $\psi_{k,m+1}(y_k^{(m)})=y_k^{(m)}$ and \eqref{intersec_halfsp} follows from the inductive hypothesis.  Otherwise, we have
\begin{equation}\label{eqmpu}
 y_k^{(m)}\in B^n\setminus H_{k,m+1}^+.
\end{equation}
Suppose that there is a $j\in\{1,\dots,m\}$ such that
\begin{equation}\label{eqmi}
\psi_{k,m+1}(y_k^{(m)})\in B^n\setminus H_{k,j}^+.
\end{equation}
On $B^n\setminus H_{k,m+1}^+$, $\psi_{k,m+1}$ equals $R_{H_{k,m+1}}$. Since  $H_{k,m+1}\cap(1-2\delta) B^n=\emptyset$, $\psi_{k,m+1}$ maps $B^n\setminus H_{k,m+1}^+$ into $C(w_{k,m+1}, 1-4\delta)$, so, by  \eqref{eqmpu},
$\psi_{k,m+1}(y_k^{(m)})\in C(w_{k,m+1}, 1-4\delta)$.
Moreover, $H_{k,j}\cap(1-2\delta) B^n=\emptyset$, yielding $B^n\setminus H_{k,j}^+\subset C(w_{k,j}, 1-4\delta)$ and hence, by  \eqref{eqmi}, $ \psi_{k,m+1}(y_k^{(m)})\in C(w_{k,j}, 1-4\delta)$. Therefore $C(w_{k,j}, 1-4\delta)\cap C(w_{k,m+1}, 1-4\delta)\neq \emptyset$ and our assumptions on $\delta$ give
\begin{equation}\label{condition_wi}
 \langle w_{k,j}, w_{k,m+1}\rangle \geq 0.
\end{equation}
We have $ \psi_{k,m+1}(y_k^{(m)})=R_{H_{k,m+1}}(y_k^{(m)})= y_k^{(m)}-c\, w_{k,m+1}$ for some $c>0$.  The inductive hypothesis implies that $y_k^{(m)}\in H_{k,j}^+$, so by \eqref{condition_wi}, we obtain
$$\langle \psi_{k,m+1}(y_k^{(m)}), w_{k,j}\rangle< \langle y_k^{(m)}, w_{k,j}\rangle$$
and thus $\psi_{k,m+1}(y_k^{(m)})\in H_{k,j}^+$. This contradicts \eqref{eqmi}, so \eqref{intersec_halfsp} is true when $i=m+1$ and so for all $i=1,\dots, i_k$.

By the definition of $z_k$ and \eqref{intersec_halfsp}, we have $z_k\in H_{k,i_k}^+$.  Hence $\|y_k-z_k\|=\|w_{k,i_k}-z_k\|\ge \beta_{k,i_k}>\delta$.  To prove that $\psi_k(y_k)=\psi_k(z_k)$, it suffices to show \eqref{non_inject_ik}, and this follows from the definition of $z_k$, the fact that $z_k\in H_{k,1}^+\cap \dots \cap H_{k,i_k}^+$ in view of  \eqref{intersec_halfsp} with $i=i_k$, and the fact that on $H_{k,1}^+\cap \dots \cap H_{k,i_k}^+$, each map $\psi_{k,i}$, $i=1, \dots,i_k$, is the identity. This completes the proof of \eqref{non_inject_k} under our initial assumption \eqref{oinpositivehalfsp}.

We now prove \eqref{non_inject_k} without assuming \eqref{oinpositivehalfsp}. For $i=1,\dots,i_k$,  let $f_i:\R^n\to\R^n$ be defined by
\[
f_i=\begin{cases}
\Id, &\text{if $x_{k,i-1}\in H_{k,i}^+$,} \\
R_{H_{k,i}}, &\text{otherwise.}
\end{cases}
\]
Then $f_i$ is a bijection and $f_i^2=\Id$.
Now let $\tilde{\psi}_{k,1}=f_1\circ\psi_{k,1}$ and for $i=2,\dots,i_k$, let
$$\tilde{\psi}_{k,i}=f_1\circ\dots \circ f_{i}\circ\psi_{k,i}\circ f_{i-1}\circ\dots \circ f_1.$$
For each $i$, $\tilde{\psi}_{k,i}: B^n\to B^n$ acts on $B^n$ as the contraction associated with the polarization with respect to the hyperplane $f_1\circ\dots\circ f_{i-1}H_{k,i}$, oriented so that $o$ belongs to the positive half-space. We may therefore replace $\psi_{k,i}$ by $\tilde{\psi}_{k,i}$ in the previous argument and
conclude that there are $y_k, z_k\in B^n$ such that $\liminf_{k\to\infty}\|y_k-z_k\|>0$ and
\[
(\tilde{\psi}_{k,i_k}\circ\dots\circ\tilde{\psi}_{k,1})(y_k)=
(\tilde{\psi}_{k,i_k}\circ\dots\circ\tilde{\psi}_{k,1})(z_k).
\]
Since
\[
\tilde{\psi}_{k,i_k}\circ\dots\circ\tilde{\psi}_{k,1}=f_1\circ\dots\circ f_{i_k}\circ\psi_{k,i_k}\circ\dots\circ\psi_{k,1},
\]
and $f_1\circ\dots\circ f_{i_k}$ is a bijection, \eqref{non_inject_ik} holds too. This concludes the proof of \eqref{non_inject_k}.

By taking subsequences, if necessary, we may assume that $y_k\to y\in B^n$ and $z_k\to z\in B^n$ as $k\to\infty$, where $y\neq z$ by \eqref{non_inject_k}. Since $\psi_k\to \psi$ uniformly as $k\to\infty$, by Lemma~\ref{lem11824}, \eqref{non_inject_k} implies that $\psi(y)=\psi(z)$, so $\psi$ is not injective on $B^n$.
\end{proof}

\begin{cor}\label{cor42}
If $H$ is a hyperplane in $\R^n$, the Brock set map $B_H$ with parameter $0<b<1$ cannot be approximated on ${\cK}^n_n$ by finite compositions of polarizations.
\end{cor}

\begin{proof}
The contraction $\psi_{B_H}$ from \eqref{eq:psi1} associated with $B_H$ is defined by \eqref{f4}.  This map is injective on $B^n$ and $\psi_{B_H}(B^n)$ is an ellipsoid of volume $b\,\kappa_n$.  It follows from Theorem~\ref{lem5} that $B_H$ cannot be approximated on ${\cK}^n_n$ by finite compositions of polarizations.
\end{proof}

We do not know if ``approximated" in Corollary~\ref{cor42} can be replaced by ``weakly approximated"; see Problem~\ref{prob4}.

\begin{lem}\label{lem_balls_close}
Let $\cE\subset \cL^n$ be a class that contains all balls, and let  $\di,\heartsuit$ be maps in $\cJ(\cE)$ with associated contractions $\psi_{\di,r},\psi_{\heartsuit,r}$, respectively, from \eqref{eq:psi1}.
For all $x\in \R^n$ and $r>0$ with $(\di B(x,r))\cap \heartsuit B(x,r)\ne\emptyset$, we have
\[
\|\psi_{\di,r}(x)-\psi_{\heartsuit,r}(x)\|\le \frac{n}{2r^{n-1}\kappa_{n-1}}\|1_{\di B(x,r)}-1_{\heartsuit B(x,r)}\|_1.
\]
\end{lem}

\begin{proof}
The quantity $\|1_{\di B(x,r)}-1_{\heartsuit B(x,r)}\|_1$ is the volume of the symmetric difference of the balls $B_1=B(\psi_{\di,r}(x),r)$ and $B_2=B(\psi_{\heartsuit,r}(x),r)$. Let $tu=\psi_{\di,r}(x)-\psi_{\heartsuit,r}(x)$, where $u\in S^{n-1}$ and $t\ge 0$. Since $(\di B(x,r))\cap \heartsuit B(x,r)\ne\emptyset$, we have  $t\le 2r$. Excluding a trivial case, we may assume that $t>0$.  For $i=1,2$, let $H_i$ be the hyperplane orthogonal to $u$ and containing the center of $B_i$.  Then $\conv(B_1\cup B_2)=C\cup E_1\cup E_2$, where $C$ is a spherical cylinder $C$ of radius $r$ and length $t$ bounded by $H_1$ and $H_2$, and $E_i$ is the half-ball of radius $r$ contained in $B_i$, bounded by $H_i$, and such that $C\cap\inte E_i=\emptyset$, $i=1,2$.  Let $v$ be the midpoint of the line segment joining the centers of $B_1$ and $B_2$, and let $C_i$ be the spherical cone of radius $r$ and height $t/2$ with apex $v$ and base $B_i\cap H_i$, $i=1,2$.  It is easy to see that
$$B_1\setminus B_2=(B_2+tu)\setminus B_2\supset ((B_2+t[o,u])\setminus B_2)\setminus (C\setminus (C_1\cup C_2)),$$
and that by Cavalieri's principle, $\cH^n((B_2+t[0,u])\setminus B_2)=\cH^n(C)$.  Consequently,
$$
\cH^n(B_1\setminus B_2)\ge  \cH^n((B_2+t[0,u])\setminus B_2)-(\cH^n(C)-2\cH^n(C_1))=2\cH^n(C_1),
$$
and similarly $\cH^n(B_2\setminus B_1)\ge 2\cH^n(C_1)$.  Therefore
\begin{align*}
\|1_{\di B(x,r)}-1_{\heartsuit B(x,r)}\|_1&=\cH^n(B_1\triangle B_2)\\
&\ge 4\cH^n(C_1)=\frac{4}{n}r^{n-1}\kappa_{n-1}\frac{t}{2}
=\frac{2r^{n-1}\kappa_{n-1}}{n }\|\psi_{\di,r}(x)-\psi_{\heartsuit,r}(x)\|,
\end{align*}
as required.
\end{proof}

The following lemma generalizes the observation that a contraction $\psi:\R^n\to \R^n$ that does not decrease the distance between two given points acts as an affine function on the line segment with these points as endpoints. This is the special case $\ee=0$ of the lemma.

\begin{lem}\label{lem_almost_affine}
Let $\psi:\R^n\to \R^n$ be a contraction and let $\ee\ge 0$. If $x,x'\in \R^n$ are such that $L=\|\psi(x)-\psi(x')\|$ satisfies $\|x-x'\|\le L+\ee$, then
\[
\|\psi((1-t)x+tx')-((1-t)\psi(x)+t\psi(x'))\|
\le \sqrt{\ee(2L+\ee)}, \qquad 0\le t\le 1.
\]
\end{lem}

\begin{proof}
Let $w=(1-t)x+tx'$ for $0\le t\le 1$. Since $\psi$ is a contraction, the assumption $\|x-x'\|\le L+\ee$ implies that
\[
\|\psi(w)-\psi(x)\|\le \|w-x\|=t\|x-x'\|\le t(L+\ee)
\]
and
\[
\|\psi(w)-\psi(x')\|\le \|w-x'\|=(1-t)\|x-x'\|\le (1-t)(L+\ee).
\]
Hence,
\[
\psi(w)\in B(\psi(x),t(L+\ee))\cap B(\psi(x'),(1-t)(L+\ee))=M,
\]
say.  Direct calculation shows that $(1-t)\psi(x)+t\psi(x')\in M$.
We claim that $M$ is contained in a ball with radius $\rho=(1/2)\sqrt{\ee(2L+\ee)}$, yielding
\begin{equation}\label{eq:close}
\|\psi(w)-((1-t)\psi(x)+t\psi(x'))\|\le 2\rho=\sqrt{\ee(2L+\ee)},
\end{equation}
as desired.

To prove the claim, we may assume without loss of generality that $\psi(x)=o$ and $\psi(x')=Le_1$, and set $r=(1-t)(L+\ee)$ and $R=t(L+\ee)$.  The smallest ball containing $M=B(o,R)\cap B(Le_1,r)$ is $B(o,R)$, if $r^2\ge R^2+L^2$, and $B(Le_1,r)$, if $R^2\ge L^2+r^2$.  These two conditions are equivalent to $t\le t_0$ and $t\ge t_1$, respectively, where
$$t_0= \frac{(L+\ee)^2-L^2}{2(L+\ee)^2}\le\frac12~\quad{\text{and}}~\quad
t_1= \frac{(L+\ee)^2+L^2}{2(L+\ee)^2}\ge\frac12.$$
When $t_0\le t\le t_1$, the spheres $\partial B(o,R)$ and $\partial B(Le_1,r)$ intersect, and the radius $a$ of this intersection is also the radius of the smallest ball containing $M$.  To find $a$, one can appeal to a known formula (see \cite{W}) for the radius of the circle of intersection of the spheres $\partial B(o,R)$ and $\partial B(de_1,r)$ in $\R^3$, namely,
$$\frac{1}{2d}\left((-d+r-R)(-d-r+R)(-d+r+R)(d+r+R)\right)^{1/2}.$$
(The calculation is the same for all $n\ge 2$.)  Substituting for $R$ and $r$ and setting $d=L$, we obtain
\begin{equation}\label{eq821241}
a=\frac{1}{2L}\left((-L+(1-2t)(L+\ee))(-L+(2t-1)(L+\ee))
\ee(2L+\ee)\right)^{1/2}.
\end{equation}
Summarizing, the radius of the smallest ball containing $M$ is the continuous function of $t$ equal to \eqref{eq821241}, if $t_0\le t\le t_1$, to $t(L+\ee)$, if $t\le t_0$, and to $(1-t)(L+\ee)$, if $t\ge t_1$.  By differentiation, we find that the maximum value of $a$ in \eqref{eq821241} occurs when $t=1/2$, and equals $(1/2)\sqrt{\ee(2L+\ee)}$.  The maxima of $t(L+\ee)$ for $t\le t_0$ and of $(1-t)(L+\ee)$ for $t\ge t_1$ are
$$t_0(L+\ee)=(1-t_1)(L+\ee)=\frac{\ee(2L+\ee)}{2(L+\ee)}\le \frac{\sqrt{\ee(2L+\ee)}}{2},$$
since $L\ge 0$.  Thus $\rho=(1/2)\sqrt{\ee(2L+\ee)}$, proving \eqref{eq:close}.
\end{proof}

In the following result it is crucial that the approximating maps have associated contractions from \eqref{eq:psi1} that are independent of $r>0$. This is satisfied if these maps are smoothing, for instance. Recall the definitions of the classes $\cJ(\cE)$ and $\cJ_H(\cE)$, where $\cE\subset {\cL}^n$, from Section~\ref{Properties}, and that $\cF_n^n$ denotes the family of all finite unions of balls in $\R^n$.

\begin{thm}\label{thm:not_weaklyfiniteBalls}
Let $H$ be a hyperplane in $\R^n$ and let $\di\in \cJ_H(\cF_n^n)$. Furthermore, let $\cJ'\subset \cJ(\cF_n^n)$ be such that for each map in $\cJ'$, the associated contraction from \eqref{eq:psi1} is independent of $r>0$.

If $\di$ is weakly approximable on ${\cF}_n^n$ by maps in $\cJ'$, then there is a sequence $(\di_k)$ from $\cJ'$ such that the associated contractions $\psi_{\di_k}$ converge uniformly on $B^n$ to the contraction $\psi_{\di}$ associated with $\di$.
\end{thm}

\begin{proof}
We shall ignore sets of measure zero in the proof. Let $\ee>0$. It will suffice to show that there is a map $\heartsuit\in\cJ'$ such that its associated contraction $\psi_{\heartsuit}$ from \eqref{eq:psi1} satisfies $\|\psi_{\heartsuit}(x)-\psi_\di(x)\|\le \ee$ for all $x\in B^n$.

Since $\di$ respects $H$-cylinders, we have $(\di B^n)|H=
B(\psi_\di(o),1)|H\subset B^n+H^{\perp}$, so $\psi_\di(o)\in H^\perp$. By applying a translation by $-\psi_\di(o)$ to the entire construction, if necessary, we may assume that $\psi_\di(o)=o$. We may also assume that $H=e_n^\perp$.  Then from the second relation in \eqref{eq:omregn}, we have $\varphi_{\di}(0)=\langle \psi_{\di}(o),e_n\rangle=\langle o,e_n\rangle=0$, where $\varphi_{\di}$ is the contraction associated with $\di$ from \eqref{ball}.

We start by constructing a set $K\in \cF_n^n$ to which the weak approximation property will be applied. Let $2\le m\in \N$ satisfy $\sqrt{n-1}/(2m)\le \ee/4$, and consider the family $\{R_1,\ldots,R_{m'}\}$ of $m'=2(2m+1)^{n-1}$ rays
\[
\{(t,i_2/m,\ldots,i_n/m):t\ge 0\},\qquad
\{(t,i_2/m,\ldots,i_n/m):t\le 0\},
\]
$i_2,\ldots,i_n=-m,\dots,m$, each emanating from the hyperplane $e_1^{\perp}$ and parallel to $e_1$.  The order of enumeration is irrelevant here, apart from the convention that $R_{2i-1}\cup R_{2i}$ is a line for $i=1,\dots,m'/2$. For each ray, we now define iteratively a ball with center on that ray.

Let $B_{1}=B(x_1,r_1)$, where $x_1\in R_1$, $r_1=1+1/m'$, and where $B_{1}$ and the cube $[-1,1]^n$ meet only on their boundaries.  Suppose that $i=1,\dots,m'-1$ and we have defined balls $B_j$, $j=1,\dots,i$ such that $B_j$ has its center on $R_j$.  Let $C_i=\rho_i B^n+\myspan\{e_n\}$ be the smallest cylinder with axis parallel to $e_n$ containing the balls $B_1,\dots,B_i$, and let $B_{i+1}=B(x_{i+1},r_{i+1})$, where $x_{i+1}\in R_{i+1}$, $r_{i+1}=1+(i+1)/m'$, and where $B_{i+1}$ and the enlarged cylinder $C_i+4B^n$ meet only on their boundaries.

Let $C=\rho B^n+\myspan\{e_n\}$ be the smallest cylinder with axis parallel to $e_n$ containing the balls $B_1,\dots,B_{m'}$. Let $S_1,\dots,S_{2(n-1)}$ be the rays emanating from the origin and spanned by $\pm e_1,\dots,\pm e_{n-1}$, and for $j=1,\dots,2(n-1)$, let $D_j=B(y_j,(n-1)^{-1/2})$, where $y_j\in S_j$ and where $D_j$ and the enlarged cylinder $C+4B^n$ meet only on their boundaries.

Define
\[
K_0= {\textstyle \bigcup}_{j=1}^{2(n-1)} D_j ~~\quad{\text{and}}~~\quad
K=K_0\cup {\textstyle \bigcup}_{i=1}^{m'} B_i.
\]
The set $K$ is a finite disjoint union of balls with radii in $\{(n-1)^{-1/2}\}\cup[1,2]$. It is easily checked that by construction, if $K_1$ and $K_2$ are different components (balls) of $K$, then
\begin{equation}\label{eq:disjointn}
d(K_1|H,K_2|H)\ge 4,
\end{equation}
where $d(A,B)=\inf\{\|a-b\|: a\in A, b\in B\}$ is the usual distance between sets. Let $B_c=B(o,\rho_c)$ be the circumball of $K$ and note that all the balls in $K$ are contained in its interior, with the exception of the balls $D_j$, and that $D_j\cap\partial B(o,\rho_c)=S_j\cap \partial B(o,\rho_c)$, $j=1,\dots,2(n-1)$.

Next, we determine $\di K$. Since $\di\in \cJ_H(\cF_n^n)$ and $H=e_n^{\perp}$, relation
\eqref{eq:omregn} shows that if $x\in e_n^{\perp}$, then $\psi_{\di}(x)=x+\varphi_{\di}(0)e_n=x$. It follows from \eqref{eq:psi1} that $\di D_j=D_j$, $j=1,\dots, 2(n-1)$. Therefore $D_j=\di D_j\subset \di K_0$ by monotonicity, which gives $K_0\subset \di K_0$ and hence $\di K_0=K_0$ by the measure-preserving property of $\di$.

If $B=B(x,r)$ is a component of $K$, then $\di B\subset (B|H)+H^\perp$ as $\di$ respects $H$-cylinders, so \eqref{eq:disjointn} yields
\begin{equation}\label{eq:disjointn_din}
d\big((\di K_1)|H,(\di K_2)|H\big)\ge 4
\end{equation}
for different components $K_1, K_2$ of $K$. Consequently, the images under $\di$ of the components of $K$ are disjoint balls. In particular, \eqref{eq:psi1} and \eqref{eq:omregn} imply that the components of $\di \left(\cup_{i=1}^{m'} B_i\right)$ are the disjoint balls
$B\big(x_i^\di,r_i\big)$, with
\[
x_i^\di=x_i+\left(\varphi_\di(\langle x_i,e_n\rangle)-\langle x_i,e_n\rangle\right)e_n
\]
for $i=1,\dots,m'$, and $\di K=K_0\cup\cup_{i=1}^{m'} B(x_i^\di,r_i).$

Choose $\ee_1>0$ such that
\begin{equation}\label{eqepchoose}
\ee_1<\min\left\{1,\frac{{\cH}^n(D_1)}{2}, \frac{5\kappa_{n-1}}{n}, \frac{n\kappa_n}{m'}, \frac{(\kappa_{n-1}\ee)^2}{50n^2\rho_c}\right\}.
\end{equation}
(The list of quantities in the minimum is for convenience; in fact, one can show that the second quantity is less than the third.) By assumption, there is a map $\heartsuit\in \cJ'$ such that
$\|1_{\di K}-1_{\heartsuit K}\|_1\le \ee_1$.

The circumball $B_c=B(o,\rho_c)$ of $K$ has its center at the origin, so $\psi_\di(o)=o$ gives $\di B_c=B_c$. Since $\heartsuit K\subset \heartsuit B_c=B(\psi_{\heartsuit}(o),\rho_c)$ we obtain
\begin{align*}
\cH^n\left(D_j\setminus B(\psi_{\heartsuit}(o),\rho_c)\right)
&=\cH^n\left((\di D_j)\setminus \heartsuit B_c\right)\\
&\le \cH^n((\di K)\setminus \heartsuit K)\le \|1_{\di K}-1_{\heartsuit K}\|_1\le \ee_1<{\cH}^n(D_j)/2
\end{align*}
for $j=1,\dots,2(n-1)$.  It follows that $B(\psi_{\heartsuit}(o),\rho_c)$ contains the centers of all the balls in $K_0$, that is, the points $\pm (\rho_c-(n-1)^{-1/2}) e_j$, for $j=1,\dots,n-1$.  Thus each vector $\psi_{\heartsuit}(o)\pm (\rho_c-(n-1)^{-1/2}) e_j$, $j=1,\dots,n-1$, has Euclidean norm less than or equal to $\rho_c$.  By considering their projections on the $j$th coordinate axis, we see that
$$|\langle \psi_{\heartsuit}(o), e_j\rangle|\le \frac1{\sqrt{n-1}}$$
for $j=1,\dots,n-1$.  From this and $\psi_{\heartsuit}(o)|H=\sum_{j=1}^{n-1}\langle \psi_{\heartsuit}(o),e_j\rangle e_j$, we obtain
\begin{equation}\label{eq:circumcentern}
\|\psi_{\heartsuit}(o)|H\|\le 1.
\end{equation}
Choose $i\in \{1,\ldots, m'\}$. The ball $\heartsuit B_i$ must meet the interior of at least one component of $\di K$, as otherwise, we would have the contradiction
\begin{eqnarray*}
{\cH}^n(B_i) = {\cH}^n(\heartsuit B_i)&=&\cH^n\left( (\heartsuit B_i)\setminus \di K\right)\\
&\le& \cH^n\left( (\heartsuit K)\setminus \di K\right)
\le \|1_{\di K}-1_{\heartsuit K}\|_1\le
\ee_1< {\cH}^n(D_1)/2\le {\cH}^n(B_i)/2.
\end{eqnarray*}

We claim that $\di B_i$ is the unique component of $\di K$ that meets $\heartsuit B_i$. To see this, first recall that $B_i$ is contained in the cylinder $C_i=\rho_i B^n+\myspan\{e_n\}$ and that $|\langle x_i,e_n\rangle|\le 1$. This gives $B_i\subset B(o,\rho_i+1)$, since
\[
\|x_i\|+r_i\le \sqrt{(\rho_i-r_i)^2+1}+r_i\le (\rho_i-r_i)+1+r_i=\rho_i+1.
\]
Using \eqref{eq:circumcentern}, we obtain
\begin{align*}
\heartsuit B_i&\subset
\heartsuit B\left(o,\rho_i+1\right)
=B\left(\psi_{\heartsuit}(o),\rho_i+1\right)\subset
(\rho_i+2) B^n+\myspan\{e_n\}=C_i+2B^n.
\end{align*}
By the construction of $K$, the only components of $K$ contained in the cylinder $C_i+2B^n$ are $B_1,\ldots,B_i$, and since $\di$ respects $H$-cylinders, the only components of $\di K$ contained in  $C_i+2B^n$ are  $\di B_1,\ldots,\di B_i$.

Because $\heartsuit B_i$ has radius at most 2, by \eqref{eq:disjointn_din} it meets the interior of exactly one component of $(C_i+2B^n)\cap \di K$.  Let $\di B'$ be this component, and suppose that $B'\neq B_i$. Then our construction ensures that the radius $r$ of $B'$ satisfies $r_i-r\ge 1/m'$.  Thus, using $r_i,r\ge 1$ and $(\heartsuit B_i)\cap \di K=(\heartsuit B_i)\cap \di B'$, we find that
\begin{align*}
\frac{n\kappa_n}{m'}&\le  \kappa_n(r_i-r)\sum_{k=0}^{n-1}r_i^kr^{n-1-k}= \kappa_n(r_i^n-r^n)
=\cH^n(\heartsuit B_i)-\cH^n(\di B')
\le \cH^n\left( (\heartsuit B_i)\setminus\di B'\right)\nonumber\\
& = \cH^n((\heartsuit B_i)\setminus \di K)\le \cH^n\left( (\heartsuit K)\setminus\di K\right)\le \|1_{\di K}-1_{\heartsuit K}\|_1\le \ee_1,
\end{align*}
contradicting the choice \eqref{eqepchoose} of $\ee_1$. Therefore, $(\di B_i)\cap \heartsuit B_i\neq\emptyset$, as claimed.

Since $\di B_i$ meets $\heartsuit B_i$, Lemma~\ref{lem_balls_close} with $x=x_i$ and $r=r_i$ implies that
\begin{equation}\label{first_fixed_pointn}
\left\|\psi_{\di}(x_i)-\psi_{\heartsuit}(x_i)\right\|
\le \frac{n}{2r_i^{n-1}\kappa_{n-1}}\|1_{\di B_i}-1_{\heartsuit B_i}\|_1
\le\frac{n}{2\kappa_{n-1}} \|1_{\di K}-1_{\heartsuit K}\|_1\le \frac{n\ee_1}{2\kappa_{n-1}}.
\end{equation}

For $i=1,\dots,m'/2$, the points $x=x_{2i-1}$ and $x'=x_{2i}$ belong to the line $\ell_i=R_{2i-1}\cup R_{2i}$, which is parallel to $H$. Since $\di$ respects $H$-cylinders, \eqref{eq:omregn} implies that $\|\psi_\di(x)-\psi_\di(x')\|=\|x-x'\|$, so $\psi_\di$ is affine when restricted to $[x,x']$.  The last equality and \eqref{first_fixed_pointn} yield
\begin{equation}\label{eq:rho-1}
\|\psi_{\heartsuit}(x)-\psi_{\heartsuit}(x')\|\ge \|x-x'\|-\frac{n\ee_1}{\kappa_{n-1}}.
\end{equation}
Let $z=(1-t)x+tx'$, where $0<t<1$.  Then \eqref{eq:rho-1} allows us to apply Lemma~\ref{lem_almost_affine} with $\ee=n\ee_1/\kappa_{n-1}$ and $\psi=\psi_{\heartsuit}$, and \eqref{first_fixed_pointn}, to obtain
\begin{eqnarray}
\left\|	\psi_{\heartsuit}(z)-\psi_{\di}(z)\right\|
&=& \left\|	\psi_{\heartsuit}(z)-((1-t)\psi_{\di}(x)+t\psi_{\di}(x'))\right\|\nonumber\\
&\le& \left\|\psi_{\heartsuit}(z)-((1-t)\psi_{\heartsuit}(x)
+t\psi_{\heartsuit}(x'))\right\|+\nonumber\\
& &+(1-t)\left\|\psi_{\heartsuit}(x)-\psi_{\di}(x)\right\|+
t\left\|\psi_{\heartsuit}(x')-\psi_{\di}(x')\right\|\nonumber\\
&\le &\sqrt{\frac{n\ee_1}{\kappa_{n-1}}\left(2L+\frac{n\ee_1}{\kappa_{n-1}}
\right)}+\frac{n\ee_1}{2\kappa_{n-1}}.
\label{atlastn}
\end{eqnarray}
By \eqref{eq:disjointn}, we have $\|x-x'\|\ge 6$, and the choice \eqref{eqepchoose} of $\ee_1$ gives $n\ee_1/\kappa_{n-1}\le 5$.  These together imply the second inequality in the following chain:
\[
\ee_1\le 1\le \|x-x'\|-\frac{n\ee_1}{\kappa_{n-1}}\le L=\|\psi_{\heartsuit}(x)-\psi_{\heartsuit}(x')\|\le \|x-x'\|\le 2\rho_c.
\]
The other relations follow from the choice \eqref{eqepchoose} of $\ee_1$, \eqref{eq:rho-1}, and the fact that $\psi_\heartsuit$ is a contraction.  It is straightforward to check that for integer $n\ge 2$, we have $n/\kappa_{n-1}\ge 3/\pi$, the value when $n=3$ or 4, and hence $2\le (3n)/\kappa_{n-1}$.  All these estimates,  \eqref{atlastn}, and the choice \eqref{eqepchoose} of $\ee_1$ yield
\begin{equation}\label{eq6425}\left\|	\psi_{\heartsuit}(z)-\psi_{\di}(z)\right\|\le\sqrt{\frac{n\ee_1}{\kappa_{n-1}}\left(\frac{3nL}{\kappa_{n-1}}
+\frac{nL}{\kappa_{n-1}}\right)}+\frac{n\sqrt{L\ee_1}}{2\kappa_{n-1}}
= \frac{5n\sqrt{L\ee_1}}{2\kappa_{n-1}}\le \frac{5n\sqrt{2\rho_c\ee_1}}{2\kappa_{n-1}}\le\frac{\ee}{2}.
\end{equation}

By construction, $[-1,1]^n\cap \ell_i\subset [x, x']$, so \eqref{eq6425} holds for all $z\in [-1,1]^n\cap \cup_{i=1}^{m'/2}\ell_i$. Now for any $w\in B^n$, there is an $i\in\{1,\dots,m'/2\}$ and a $z\in [-1,1]^n\cap\ell_i$ with $\|w-z\|\le \sqrt{n-1}/(2m)\le \ee/4$.  Since $\psi_{\heartsuit}$ and $\psi_\di$ are contractions, \eqref{eq6425} gives
\begin{align*}
\left\|	\psi_{\heartsuit}(w)-\psi_\di (w)\right\|
&\le \left\|\psi_{\heartsuit}(w)-\psi_{\heartsuit} (z)\right\|
+\left\|\psi_{\heartsuit} (z)-\psi_{\di}(z)\right\|+
\left\|\psi_{\di}(z)-\psi_\di (w)\right\|\\
&\le \left\|\psi_{\heartsuit}(z)-\psi_\di (z)\right\|
+2\left\|w-z\right\|\le \ee,
\end{align*}
as required.
\end{proof}

\begin{cor}\label{corBrockWeak}
Let $\cF_n^n\subset \cE\subset \cL^n$. If $H$ is a hyperplane in $\R^n$, the Brock set map $\di_{B_H}$ with parameter $0<b<1$ cannot be weakly approximated on $\cE$ by finite compositions of polarizations.
\end{cor}

\begin{proof}
Consider the class of finite compositions of polarizations (not necessarily with respect to parallel hyperplanes). The associated contractions from \eqref{eq:psi1} of polarizations are independent of $r>0$ and the associated contractions of their compositions inherit this property.

If $\di_{B_H}$ could be weakly approximated on $\cE$ by finite compositions of polarizations, it could a fortiori be weakly approximated on $\cF_n^n$ by their restrictions to $\cF_n^n$. Theorem~\ref{thm:not_weaklyfiniteBalls} would then yield a sequence $(\di_k)$ of finite compositions of polarizations such that the associated contractions $\psi_{\di_k}$ converge uniformly on $B^n$ to the contraction $\psi_{\di_{B_H}}$ associated with $\di_{B_H}$. Since  $\psi_{\di_{B_H}}$ is injective, Theorem~\ref{lem5} would imply that ${\cH}^n(\psi_{\di_{B_H}}(B^n))=\kappa_n$, a contradiction, as $b<1$.
\end{proof}

\begin{cor}\label{corBrockWeakfunctions}
Let $X\subset \cV(\R^n)$ contain all nonnegative $C^{\infty}$ functions with compact support, and let $1\le p<\infty$. If $H$ is a hyperplane in $\R^n$, the Brock rearrangement $B_H$ with parameter $0<b<1$ is not weakly approximable in $L^p$ on $X$ by finite compositions of polarizations.
\end{cor}

\begin{proof}
Let $1\le p<\infty$ and suppose that $B_H$ is weakly approximable in $L^p$ on $X$ by finite compositions of polarizations. As $X$ contains all nonnegative $C^{\infty}$ functions with compact support, the usual approximation argument based on mollifiers shows that $\cl_p X=L_+^p(\R^n)$.  By Lemma~\ref{lemLpclosure} (with $T=B_H$ and $\cW$ the class of finite compositions of polarizations), $B_H$ is weakly approximable in $L^p$ on $L_+^p(\R^n)$ by finite compositions of polarizations.  Then Theorem~\ref{thm514} implies that $\di_{B_H}$ is weakly approximable on $\mathcal C^n$ by finite compositions of polarizations.  Since ${\cF}^n_n\subset\mathcal C^n$, this contradicts  Corollary~\ref{corBrockWeak}.
\end{proof}

\section{Approximating the Solynin set map}\label{Solynin}

Here we aim to show that if $H$ is a hyperplane in $\R^n$, then the Solynin rearrangement $So_H$ with respect to $H$ is approximable in $L^p$ on $L^p_+(\R^n)$ by finite compositions of polarizations.  Most of the work will be proving that the associated set map $\di_{So_H}$ is approximable on ${\mathcal{C}}^n$ by finite compositions of polarizations.  Since both the Solynin set map and polarization with respect to $H$ act fiber-wise on lines orthogonal to $H$, it suffices to prove the latter result when $n=1$ and $H=\{o\}$, with the usual orientation of $\R$.

We shall need some notation and lemmas. If $a\in \R$, let $P_a$ denote polarization in $\R$ with respect to $a$ and the usual orientation of $\R$. For $A\subset \R$, $m\in \N$, $k\in\Z$, and $-2^{2m}\leq k\leq 0$, let
\[
 A_{m,k}=P_{k/2^m}\circ P_{(k-1)/2^m}\circ\dots\circ P_{-2^m+1/2^m}\circ P_{-2^m}A.
\]

Regarding the existence of $t$ in part (ii) of the following lemma, we observe that if $x\in A_{m,j}$ and $x\geq0$ then $x\in A_{m,i}$ for all $i\geq j$. Since $x\notin A$ and $x\in A_{m,0}$, the existence of such a $t$ follows.

\begin{lem}\label{welldistributed}
\noindent{\rm{(i)}}  Assume $x\in A_{m,0}$,  $x\in(k/2^m, (k+1)/2^m]$ for some $k\in\Z$, $-2^{2m}-1<k\leq-1$. The set $A_{m,0}$ also contains
\begin{equation}\label{welldist-keven}
x+\frac{2}{2^m}, x+\frac{4}{2^m},\dots, x-\frac{k}{2^m}, -x+\frac{k+2}{2^m}, -x+\frac{k+4}{2^m}, \dots,  -x
\end{equation}
when $k$ is even, and
\begin{equation}\label{welldist-kodd}
x+\frac{2}{2^m},x+\frac{4}{2^m},\dots, x-\frac{k+1}{2^m}, -x+\frac{k+1}{2^m}, -x+\frac{k+3}{2^m}, \dots, -x
\end{equation}
when $k$ is odd. In particular, for $k$ even and $s=2, 4,\dots,-(k+2)$,
\begin{equation}\label{inclusions-}
A_{m,0}\supset \left(A_{m,0}\cap\left(\frac{k}{2^m},\frac{k+2}{2^m}\right]\right) + \frac{s}{2^m}, \quad  A_{m,0}\supset-\left(A_{m,0}\cap\left(\frac{k}{2^m},\frac{k+2}{2^m}\right]\right) - \frac{s}{2^m}.
\end{equation}

\noindent{\rm{(ii)}} Assume $x\in A_{m,0}\setminus A$ and $x\in[k/2^m, (k+1)/2^m)$ for some $k\in\Z$, $k\geq 1$. Let $t=t(x,m)\in\Z$, $t\le 0$, be the first index such that $x\in A_{m,t}$ and assume $2t-k>-2^{2m}$.  The set $A_{m,0}$ also contains
\begin{equation}\label{welldist+keven}
  -x+\frac{2}{2^m},-x+\frac{4}{2^m},\dots,-x+\frac{k}{2^m},
 x-\frac{k}{2^m}, x-\frac{k-2}{2^m} \dots,  x-\frac2{2^m}
\end{equation}
when $k$ is even, and
\begin{equation}\label{welldist+kodd}
  -x+\frac{2}{2^m},-x+\frac{4}{2^m},\dots,-x+\frac{k+1}{2^m},
 x-\frac{k-1}{2^m}, x-\frac{k-3}{2^m} \dots,  x-\frac2{2^m}
\end{equation}
when $k$ is odd. In particular, if $k\ge 2$ is even and the condition on $t=t(x,m)$ above holds uniformly for any $x\in(A_{m,0}\setminus A)\cap[k/2^m, (k+1)/2^m)$ and for any $x\in(A_{m,0}\setminus A)\cap[(k+1)/2^m, (k+2)/2^m)$, then, for $s=2,4,\dots,k$,
\begin{equation}\label{inclusions+}
A_{m,0}\supset \left(\big(A_{m,0}\setminus A\big)\cap\left(\frac{k}{2^m},\frac{k+2}{2^m}\right]\right) - \frac{s}{2^m}, \quad  A_{m,0}\supset-\left(\big(A_{m,0}\setminus A\big)\cap\left(\frac{k}{2^m},\frac{k+2}{2^m}\right]\right) + \frac{s}{2^m}.
\end{equation}
\end{lem}

\begin{proof}
We will repeatedly use the following elementary properties of polarization, valid for any $a,b,y\in\R$.
\begin{enumerate}[(P1)]
  \item\label{P1} If $y\in P_a A$ and $y\leq a$ then $y, 2a-y\in A$;
  \item\label{P2} If $y, 2a-y\in A$ then they also belong to $P_a A$;
  \item\label{P3} If $y\in A$ and $y\leq a$ then $2a-y\in P_a A$;
 \item\label{P4} If $y\geq a$ and $y\in P_a A\setminus A$ then $2a-y\in A$;
  \item\label{P5} $(P_a A)-b=P_{a-b}(A-b)$.
\end{enumerate}
Property P\ref{P5} deserves some explanation. We have $y\in P_a A$ if and only if either $y\geq a$ and $y\in A$, or $y\geq a$ and $2a-y\in A$, or $y\leq a$ and both $y\in A$ and $2a-y\in A$. These three conditions, with $y$ replaced by $y+b$, are equivalent to $y-b\in P_{a-b}(A-b)$.

To prove (i), we first prove that, for $i=1,\dots,-k$,
\begin{equation}\label{formula1}
 x, \frac{2(k+i)}{2^m}-x\in A_{m,k+i-1}.
\end{equation}
Indeed, when $i=-k$, the assumption $x\in A_{m,0}$ and  P\ref{P1} applied to $A_{m,0}=P_{0}A_{m,-1}$, implies $x,-x\in A_{m,-1}$. If \eqref{formula1} is valid for a given $i\ge 2$, then the inductive assumption $x\in A_{m,k+i-1}$ and P\ref{P1} applied to $A_{m,k+i-1}=P_{(k+i-1)/2^m}A_{m,k+i-2}$, implies \eqref{formula1} for $i-1$.

Now we prove that, for $j=1,\dots,-k$, both when $j=2l$ is even and when $j=2l+1$ is odd,
\begin{equation}\label{formula2}
 x, x+\frac{2}{2^m},\dots,x+\frac{2l}{2^m},
 \frac{2(k+l+1)}{2^m}-x, \frac{2(k+l+2)}{2^m}-x,\dots, \frac{2(k+j)}{2^m}-x \in A_{m,k+j}.
\end{equation}
Again we argue by induction. Assume $j=1$. In this case $l=0$ and \eqref{formula2} amounts to
\[
  x, \frac{2(k+1)}{2^m}-x \in A_{m,k+1},
\]
which is a consequence of \eqref{formula1} with $i=1$ and P\ref{P2}, with $A=A_{m,k}$ and $a={(k+1)/2^m}$.
Assume \eqref{formula2} valid for a given $j<-k$. We distinguish between the cases $j$ even and $j$ odd. First assume  $j=2l+1$ odd. The points $x,2(k+j+1)/2^m-x\in A_{m,k+j}$, by \eqref{formula1} with $i=j+1$, and are symmetric with respect to $(k+j+1)/2^m$. Thus, by P\ref{P2}, they also belong to $A_{m,k+j+1}$. Similarly the points
\[
 x+\frac{2t}{2^m}\quad\text{ and } \quad \frac{2(k+j-t+1)}{2^m}-x, \quad\text{for $t=1,\dots,l$,}
\]
belong to $A_{m,k+j}$, by the inductive hypothesis, are symmetric with respect to $(k+j+1)/2^m$ and P\ref{P2} implies that they also belong to $A_{m,k+j+1}$. When $t=l$,
\begin{equation}\label{formula3}
\frac{2(k+j-t+1)}{2^m}-x= \frac{2(k+(l+1)+1)}{2^m}-x,
\end{equation}
and thus to prove \eqref{formula2} for $j+1=2(l+1)$ it remains only to prove that $x+2(l+1)/{2^m}\in A_{m,k+j+1}$. This follows from P\ref{P3} with  $a=(k+j+1)/2^m$ and the role of $y$ played by  $2(k+l+1)/2^m-x$. Indeed, $2(k+l+1)/2^m-x\leq a$ since $x>k/2^m$, and
\[
 \frac{2(k+j+1)}{2^m}-\left(\frac{2(k+l+1)}{2^m}-x\right)=x+\frac{2(l+1)}{2^m}.
\]
This concludes the inductive step when $j=2l+1$ is odd. The case $j=2l$ even is done similarly, with the only difference that \eqref{formula3} becomes, when $t=l$,
\[
 \frac{2(k+j-t+1)}{2^m}-x= \frac{2(k+l+1)}{2^m}-x,
\]
and that the last step of the previous argument is not needed.

When $j=-k$, \eqref{formula2} implies that  $A_{m,0}$ contains
\begin{equation}\label{case_xnegative}
x+\frac{2}{2^m},\dots, x+\frac{2l}{2^m}, \frac{2(k+l+1)}{2^m}-x, \frac{2(k+l+2)}{2^m}-x, \dots, -\frac{2}{2^m}-x, -x.
\end{equation}
When $k$ is even or odd,  \eqref{case_xnegative} becomes \eqref{welldist-keven} or \eqref{welldist-kodd}, respectively. The last conclusion of part (i) is a consequence of these formulas applied to $x\in A_{m,0}\cap(k/2^m, (k+1)/2^m]$ and to $x\in A_{m,0}\cap((k+1)/2^m, (k+2)/2^m]$.

Next, we prove (ii).  The assumption $x\in A_{m,t}$ implies that
\begin{equation}\label{reflected_x}
 x\in A_{m,t},\quad x\notin A_{m,t-1} \quad\text{and}\quad \frac{2t}{2^m}-x\in A_{m,t-1},
\end{equation}
where since $x\ge 0\ge t/2^m$, the last condition follows by P\ref{P4} with $A=A_{m,t-1}$ and $a=t/2^m$.
Property P\ref{P5} can be iterated and it implies, for $j\in\Z$,  $j>-2^{2m}-t+1$,
\begin{multline}\label{translation}
\big(P_{(j+t-1)/{2^m}}\circ P_{(j+t-2)/{2^m}}\circ \dots\circ P_{-2^m} A\big)-\frac{t-1}{2^m}=\\
=P_{j/2^m}\circ P_{(j-1)/2^m}\circ\dots\circ P_{-2^m-(t-1)/2^m}\left(A-\frac{t-1}{2^m}\right).
\end{multline}
We denote  the set in the right-hand side of the previous formula by
\[
 \left(A-\frac{t-1}{2^m}\right)_{*,m,j},
\]
where we put a $*$ in the subscript to distinguish this notation from the previous one, since the iteration of the polarizations starts with $P_{-2^m-(t-1)/2^m}$ instead of $P_{-2^m}$. Formula~\eqref{translation} becomes
\begin{equation}\label{translation2}
A_{m,j+t-1}-\frac{t-1}{2^m}=\left(A-\frac{t-1}{2^m}\right)_{*,m,j}.
\end{equation}
By \eqref{reflected_x} and \eqref{translation2} with $j=0$,
\[
z=\frac{2t}{2^m}-x-\frac{t-1}{2^m}=\frac{t+1}{2^m}-x\in\left(A-\frac{t-1}{2^m}\right)_{*,m,0}.
\]
We apply to $z$ the results and the methods used to prove (i). We have $z\in((t-k)/2^m,(t-k+1)/2^m]$ and $z\le 0$. The assumption on $t$  is equivalent to the condition $t-k+1>-2^{2m}-(t-1)$, needed to apply (i) with $x$ and $k$ replaced by $z$ and $t-k$, respectively. Part (i) gives, if $k-t=2l$ or $2l+1$,
$$
 z+\frac{2}{2^m},\dots,z+\frac{2l}{2^m},
 \frac{2(t-k+l+1)}{2^m}-z, \dots,  -z \in \left(A-\frac{t-1}{2^m}\right)_{*,m,0}.
$$

By using the same ideas used to prove \eqref{formula2}, the formula which yields (i), one proves, for $j=0,\dots,1-t$, with $l\in\Z$ defined by $j+k-t=2l$ or $2l+1$,
$$
 z+\frac{2j}{2^m},\dots,z+\frac{2l}{2^m},
 \frac{2(t-k+l+1)}{2^m}-z, \dots,  -z \in \left(A-\frac{t-1}{2^m}\right)_{*,m,j}.
$$
When $j=1-t$, this becomes
$$
 z+\frac{2(1-t)}{2^m},\dots,z+\frac{2l}{2^m},
 \frac{2(t-k+l+1)}{2^m}-z, \dots,  -z \in \left(A-\frac{t-1}{2^m}\right)_{*,m,1-t},
$$
where $1+k-2t=2l$ or $2l+1$.
Using \eqref{translation2}, with $j=1-t$,  and the definition of $z$, one can express the previous formula in terms of $x$ and $A_{m,0}$. We get that, if $1+k-2t=2l$ or $2l+1$,
\begin{equation}\label{case_xpositive}
 -x+\frac{2}{2^m},-x+\frac{4}{2^m},\dots,-x+\frac{2(t+l)}{2^m},
 x+\frac{2(t-k+l)}{2^m}, \dots,  x-\frac2{2^m} \in A_{m,0}.
\end{equation}
When $k$ is even, $1+k-2t$ is odd and \eqref{case_xpositive} becomes \eqref{welldist+keven}, while when $k$ is odd, $1+k-2t$ is even and \eqref{case_xpositive} becomes \eqref{welldist+kodd}.
\end{proof}

The following compactness result is essentially proved in \cite[Lemma~3]{VS2}.

\begin{prop}\label{compactness}
Let $u\in L^p(\R^n)$ be nonnegative, $1\leq p <\infty$. If, for every $m\in \N$, the function $u_m$ is obtained from $u$ via a finite number of polarizations with respect to closed halfspaces containing $o$, then $(u_m)$ is relatively compact in $L^p(\R^n)$.
\end{prop}

The condition on $p$ in the previous result is not stated in \cite{VS2}.  However, the proposition is proved by noting that the sequence $(u_m)$ satisfies the assumptions of the Riesz-Fr\'echet-Kolmogorov compactness criterion, for which the stated condition on $p$ suffices; see, for example, \cite[Corollary~4.27]{Bre}.  The result in \cite[Lemma~3]{VS2} is stated only for the particular sequence of polarizations created in \cite{VS2}, but an inspection of the proof shows that it also applies to sequences $(u_m)$ that satisfy the hypotheses of Proposition~\ref{compactness}.

\begin{lem}\label{gablem}
Let $A\subset \R$ be compact. The sequence $A_{m,0}$ converges in the $L^1$ distance, essentially, to $\di_{So_{\{o\}}A}$, the image of $A$ under the Solynin map in $\R$ with respect to $\{o\}$ and the usual orientation of $\R$.
\end{lem}

\begin{proof}
Let $m_0\in\Z$ be such that $A\subset[-2^{m_0}, 2^{m_0}]$. By Proposition~\ref{compactness}, it is enough to show that any accumulation point of the sequence $(1_{A_{m,0}})$ in $L^1$ is essentially equal to the characteristic function of $\di_{So_{\{o\}}A}$. Since a given convergent subsequence contains a sub-subsequence converging almost everywhere, the limit is the characteristic function of a bounded measurable set $B$. Clearly $|B|=|A|$.  We have to show that $B=\di_{So_{\{o\}}A}$, essentially.

Let $l\in\N$,  $l\leq m$, $i\in\Z$ even, $i<0$,  and $s=0,2, 4,\dots,-(i+2)$. We  have
\begin{equation*}
\left(A_{m,0}\cap\left(\frac{i}{2^l},\frac{i+2}{2^l}\right]\right) + \frac{s}{2^l}=
\bigcup_{\substack{i2^{m-l}\leq k \leq (i+2)2^{m-l}-2\\\text{$k$ even}}}
\left(A_{m,0}\cap\left(\frac{k}{2^m},\frac{k+2}{2^m}\right]\right) + \frac{s2^{m-l}}{2^m}.
\end{equation*}
By applying the inclusions \eqref{inclusions-} to each set in the union in the right-hand side of the above formula, we obtain the following inclusions:
\begin{equation}\label{inclusions-gen}
A_{m,0}\supset \left(A_{m,0}\cap\left(\frac{i}{2^l},\frac{i+2}{2^l}\right]\right) + \frac{s}{2^l}, \quad  A_{m,0}\supset-\left(A_{m,0}\cap\left(\frac{i}{2^l},\frac{i+2}{2^l}\right]\right) - \frac{s}{2^l}.
\end{equation}
These inclusions imply, for $i$ and $l$ as above and any even $k\in \Z$, $|k|\leq -(i+2)$,
\begin{equation}\label{inclusions-measure}
\left|A_{m,0}\cap\left(\frac{i}{2^l},\frac{i+2}{2^l}\right]\right|\leq \left|A_{m,0}\cap\left(\frac{k}{2^l},\frac{k+2}{2^l}\right]\right|
\end{equation}
In a way similar to that used for obtaining \eqref{inclusions-gen}, one proves that, if $m_0\leq l\leq m$,  $i\ge 2$ is even, and $s=2, 4,\dots,i$, we have
\begin{equation}\label{inclusions+gen}
A_{m,0}\supset \left(\big(A_{m,0}\setminus A\big)\cap\left(\frac{i}{2^l},\frac{i+2}{2^l}\right]\right) - \frac{s}{2^l}, \quad  A_{m,0}\supset-\left(\big(A_{m,0}\setminus A\big)\cap\left(\frac{i}{2^l},\frac{i+2}{2^l}\right]\right) + \frac{s}{2^l}.
\end{equation}
This comes from \eqref{inclusions+} once that we check that if $x\in[k/2^m, (k+1)/2^m)$,  for some $k\in\Z$, $k>0$, then the index $t(x,m)$ satisfies $2t-k>-2^{2m}$. To prove this, we observe that
since  $t$ is the first index such that $x\in A_{m,t}$, then $2t/2^m-x\in A_{m,t-1}$, as explained in the proof of Lemma~\ref{welldistributed} (it is a consequence of P\ref{P4}, at the beginning of the proof). But $A_{m,t-1}\subset[-2^{m_0}, 2^{m_0}]$. So we have
\[
\frac{2t}{2^m}-\frac{k}{2^m}>\frac{2t}{2^m}-x\geq -2^{m_0}\geq -2^m,
\]
which gives $2t-k>-2^{2m}$.  The inclusions \eqref{inclusions+gen} imply that, for $i$ and $l$ as in \eqref{inclusions+gen} and even $k\in\Z$, $-i\leq k\leq i-2$, we have
\begin{equation}\label{inclusions+measure}
 \left| \big(A_{m,0}\setminus A\big)\cap\left(\frac{i}{2^l},\frac{i+2}{2^l}\right]\right|
 \leq \left|A_{m,0}\cap\left(\frac{k}{2^l},\frac{k+2}{2^l}\right]\right|.
\end{equation}

Letting $m\to \infty$, the inequalities \eqref{inclusions-measure} and \eqref{inclusions+measure} pass to the limit and hold with $B$ substituting $A_{m,0}$. These inclusions imply that if $x<0$, $x\in B$ (and also if $x>0$, $x\in B\setminus A$)  is a point of density $1$ for $B$ (for $B\setminus A$, respectively),  then almost every point in $(-|x|, |x|)$ is a point of density $1$ for $B$. We know that $B\supset A\cap[0,\infty)$, because  $A_{m,0}\supset A\cap[0,\infty)$ for each $m\in\N$. These properties imply
\[
 B=[-b,b]\cup(A\cap[0,\infty)), \text{ essentially,}
\]
for a suitable $b\geq0$. Let $r_A\geq0$ be such that $\di_{So_{\{o\}} A}=[-r_A,r_A]\cup (A\cap[0,\infty))$ (see \eqref{Soldef}).  We claim that $b=r_A$. Indeed, if $b>r_A$ then
$$
 B\supset [-b, -r_A]\cup \di_{So_{\{o\}} A}, \text{ essentially,}
$$
which implies $|B|>|\di_{So_{\{o\}} A}|=|A|$. The case $b<r_A$ can be proved similarly, by exchanging the role of $B$ and $\di_{So_{\{o\}} A}$ in the previous argument.
\end{proof}

\begin{thm}\label{gabcor}
Let $H$ be a hyperplane in $\R^n$ and let $1\le p<\infty$. Then the Solynin rearrangement $So_H$ with respect to $H$ is approximable in $L^p$ on $L^p_+(\R^n)$ by finite compositions of polarizations and hence the associated set map $\di_{So_H}$ is approximable on ${\mathcal{L}}^n$ by finite compositions of polarizations.
\end{thm}

\begin{proof}
As was explained at the beginning of this section, Lemma~\ref{gablem} implies that $\di_{So_H}$ is approximable on ${\mathcal{C}}^n$ by finite compositions of polarizations.  The theorem now follows from Theorem~\ref{thm514} with $T=So_H$ and $\cW$ the class of finite compositions of polarizations.
\end{proof}

\section{Questions}\label{Questions}

\begin{prob}\label{prob522}
{\em For which rearrangements does the $L^{\infty}$ contraction property hold?  It holds for polarizations and the symmetric decreasing rearrangement, and more generally, the $(k,n)$-Steiner rearrangement; see \cite[Corollary~2.23 and Theorem~6.14]{Baer}.}
\end{prob}

\begin{prob}\label{prob2}
{\em Is there a result analogous to Theorem~\ref{thmm7} for maps $\di$ that do not necessarily respect $H$-cylinders, where the action of $\di$ on ${\cK}^n_n$ is described in terms of its associated contraction $\psi_{\di}$ from \eqref{eq:psi1}? }
\end{prob}

\begin{prob}\label{prob3}
{\em Let $\cK^n_n\subset \cE\subset \cL^n$.  Is a map $\di:\cE\to \cL^n$ which is approximable on $\cE$ by finite compositions of polarizations also weakly sequentially approximable on $\cE$ by polarizations?  Is the converse true?}
\end{prob}

\begin{prob}\label{prob4}
{\em Let $H$ be a hyperplane in $\R^n$ and let $\di_{B_H}$ be a Brock set map with parameter $0<b<1$.  Can $\di_{B_H}$ be weakly approximated on $\cK_n^n$ by finite compositions of polarizations (not necessarily all with respect to hyperplanes parallel to $H$)?}
\end{prob}

\begin{prob}\label{prob6}
{\em If $H$ is a hyperplane in $\R^n$, is $\di_{So_H}$ weakly sequentially approximable on ${\cL}^n$ by polarizations?}
\end{prob}

\begin{prob}\label{probTranslation}
{\em Does Theorem~\ref{thm514} hold when ``approximable" is replaced throughout by ``weakly sequentially approximable"?  A counterexample would follow if translations are not weakly sequentially approximable on $\cL^n$ by polarizations.}
\end{prob}


\begin{thebibliography}{999}

\bibitem{Baer}
A.~Baernstein II, {\em Symmetrization in Analysis}, with D.Drasin and R.~S.~Laugesen, Cambridge University Press, Cambridge, 2019.

\bibitem{BT}
A.~Baernstein II and B.~A.~Taylor, Spherical rearrangements, subharmonic functions, and *-functions in $n$-space, {\em Duke Math. J.} {\bf 43} (1976), 245--268.

\bibitem{BGG1}
G.~Bianchi, R.~J.~Gardner, and P.~Gronchi, Symmetrization in geometry, {\em Adv. Math.} {\bf 306} (2017), 51--88.

\bibitem{BGG2}
G.~Bianchi, R.~J.~Gardner, and P.~Gronchi,  Convergence of symmetrization processes, {\em Indiana Univ. Math. J.} {\bf 71} (2022), 785--817.

\bibitem{BGGK1}
G.~Bianchi, R.~J.~Gardner, P.~Gronchi, and M.~Kiderlen, Rearrangement and polarization, {\em Adv. Math.} {\bf 374} (2020), 107380, 51 pp.

\bibitem{BGGK2}
G.~Bianchi, R.~J.~Gardner, P.~Gronchi, and M.~Kiderlen, The P\'olya--Szeg\H{o} inequality for smoothing rearrangements, {\em J. Funct. Anal.} {\bf 287} (2024), Paper No.~110422, 56 pp.

\bibitem{Bre}
H.~Brezis, {\em Functional Analysis, Sobolev Spaces and Partial Differential Equations}, Springer, New York, 2011.

\bibitem{Bro1}
F.~Brock, Continuous Steiner-symmetrization, {\em Math. Nachr.} {\bf 172} (1995), 25--48.

\bibitem{Bro2}
F.~Brock, Continuous rearrangement and symmetry of solutions of elliptic problems, {\em Proc. Indian Acad. Sci. Math. Sci.} {\bf 110} (2000), 157--204.

\bibitem{BS}
F.~Brock and A.~Y.~Solynin, An approach to symmetrization via polarization, {\em Trans. Amer. Math. Soc.} {\bf 352} (2000), 1759--1796.

\bibitem{BF}
A.~Burchard and M.~Fortier, Random polarizations, {\em Adv. Math.} {\bf 234} (2013), 550--573.

\bibitem{Gru07}
P.~M.~Gruber, {\em Convex and Discrete Geometry}, Springer, Berlin, 2007.

\bibitem{Haj}
H.~Hajaiej, Explicit constructive approximation to symmetrization via iterated polarizations, {\em J. Convex Anal.} {\bf 17} (2010), 405--411.

\bibitem{LL}
E.~H.~Lieb and M.~Loss, {\em Analysis}, second edition, American Mathematical Society, Providence, RI, 2001.

\bibitem{McN}
A.~McNabb, Partial Steiner symmetrization and some conduction problems, {\em J. Math. Anal. Appl.} {\bf 17} (1967), 221--227.

\bibitem{Pfe}
W.~F.~Pfeffer, {\em Derivation and Integration}, Cambridge University Press, New York, 2001.

\bibitem{PS}
G.~P\'{o}lya and G.~Szeg\H{o}, {\em Isoperimetric Inequalities in Mathematical Physics}, Princeton University Press, Princeton, NJ, 1951.

\bibitem{Rog57}
C.~A.~Rogers, A single integral inequality, {\em J. London Math. Soc.} {\bf 32} (1957), 102--108.

\bibitem{Sar72}
J.~Sarvas, Symmetrization of condensers in $n$-space, {\em Ann. Acad. Sci. Fenn. Ser. A. I.} 1972, no.~522, 44~pp.

\bibitem{Sch14}
R.~Schneider, {\em Convex Bodies: The Brunn-Minkowski Theory}, second edition, Cambridge University Press, Cambridge, 2014.

\bibitem{Sol90}
A.~Y.~Solynin, Continuous symmetrization of sets, (Russian) {\em Zap. Nauchn. Sem. Leningrad. Otdel. Mat. Inst. Steklov. (LOMI)} {\bf 185} (1990), {\em Anal. Teor. Chisel i Teor. Funktsiĭ. 10}, 125--139, 186; translation in {\em J. Soviet Math.} {\bf 59} (1992), 1214--1221.

\bibitem{Sol1}
A.~Y.~Solynin, Continuous symmetrization via polarization, {\em Algebra i Analiz} {\bf 24} (2012), 157--222; reprinted in {\em St. Petersburg Math. J.} {\bf 24} (2013), 117--166.

\bibitem{Tal}
G.~Talenti, The art of rearranging, {\em Milan J. Math.} {\bf 84} (2016), 105--157.

\bibitem{VSPhD}
J.~Van Schaftingen; R\'earrangements et probl\`emes elliptiques non lin\'eaires. Thesis (Ph.D.)-Universit\'e Catholique de Louvain. 2002. 94 pp.

\bibitem{VS1}
J.~Van Schaftingen, Universal approximation of symmetrizations by polarizations, {\em Proc. Amer. Math. Soc.} {\bf 134} (2006), 177--186.

\bibitem{VS2}
J.~Van Schaftingen, Explicit approximation of the symmetric rearrangement by polarizations, {\em Arch. Math.} {\bf 93} (2009), 181--190.

\bibitem{VSW}
J.~Van Schaftingen and M.~Willem, Set transformations, symmetrizations and isoperimetric inequalities, in: {\em Nonlinear analysis and applications to physical sciences}, pp.~135--152, Springer Italia, Milan, 2004.

\bibitem{W}
Wolfram Mathworld, Sphere-sphere intersection, available online at\newline{\tt https://mathworld.wolfram.com/Sphere-SphereIntersection.html}.

\end{thebibliography}
\end{document}